\documentclass[11pt]{article}
  \usepackage{a4wide}
  \usepackage{amsmath}
  \usepackage{amssymb}
  \usepackage{latexsym}
  \usepackage[pdftex]{graphicx}
  \usepackage{subcaption}
  \usepackage{color}
  \usepackage[mathscr]{euscript}

  \newenvironment{proof}{\vspace{1ex}\noindent{\bf Proof.}}{\hspace*{\fill}$\blacksquare$\vspace{1ex}}
  \newenvironment{proofof}[1]{\vspace{1ex}\noindent{\bf Proof of #1.}}{\hspace*{\fill}$\blacksquare$\vspace{1ex}}

  \newtheorem{theorem}{Theorem} 
  \newtheorem{lemma} [theorem] {Lemma}
  \newtheorem{corollary} [theorem] {Corollary}
  \newtheorem{proposition} [theorem] {Proposition}
  \newtheorem{conjecture} [theorem] {Conjecture}
 
\newcommand{\Fcal}[0]{\ensuremath{{\mathcal F}}}

\newcommand{\Lcal}[0]{\ensuremath{{\mathcal L}}}

\newcommand{\Qcal}[0]{\ensuremath{{\mathcal Q}}}

\newcommand{\Vcal}[0]{\ensuremath{{\mathcal V}}}

\newcommand{\Xcal}[0]{\ensuremath{{\mathcal X}}}
\newcommand{\Ycal}[0]{\ensuremath{{\mathcal Y}}}
\newcommand{\Zcal}[0]{\ensuremath{{\mathcal Z}}}
\newcommand{\eR}[0]{\ensuremath{ \mathbb R}}

\newcommand{\eN}[0]{\ensuremath{ \mathbb N}}
\newcommand{\Zed}[0]{\ensuremath{ \mathbb Z}}

\newcommand{\norm}[1]{\ensuremath{\|#1\|}}

\newcommand{\Pee}[0]{\ensuremath{{\mathbb P}}}

\newcommand{\Ee}[0]{\ensuremath{{\mathbb E}}}

\newcommand{\isd}[0]{\hspace{.2ex} \raisebox{-.1ex}{$=$} \hspace{-1.5ex} 
\raisebox{1ex}{{$\scriptstyle d$}} \hspace{.8ex} }

 \newcommand{\eps}{\varepsilon}

\newcommand{\orig}{\underline{0}}

\DeclareMathOperator{\dist}{dist}

\DeclareMathOperator{\conv}{conv}
\DeclareMathOperator{\vol}{vol}

\DeclareMathOperator{\diam}{diam}
\DeclareMathOperator{\clo}{cl}
\DeclareMathOperator{\inter}{int}

\DeclareMathOperator{\Po}{Po}

\DeclareMathOperator{\dd}{d}
\DeclareMathOperator{\Var}{Var}
\DeclareMathOperator{\inr}{inr}
\DeclareMathOperator{\outr}{outr}
\DeclareMathOperator{\meanw}{meanw}
\DeclareMathOperator{\width}{width}

\definecolor{orange}{RGB}{255,127,0}
\definecolor{pink}{RGB}{255,150,150}

\newcommand{\Vtyp}[0]{\Vcal_{\text{typ}}}

\begin{document}

\title{On the shape of the typical Poisson-Voronoi cell \\ in high dimensions}

\author{%
Matthias Irlbeck\thanks{Bernoulli Institute, Groningen University, The Netherlands. E-mail: {\tt m.irlbeck@rug.nl}.}
\and
Zakhar Kabluchko\thanks{Fachbereich Mathematik und Informatik, Universit\"{a}t M\"{u}nster, Einsteinstrasse 62, 
48149 M\"{u}nster, Germany. E-mail: {\tt zakhar.kabluchko@ni-muenster.de}. ZK has been supported by the 
German Research Foundation under Germany's Excellence Strategy  EXC 2044 -- 390685587, \textit{Mathematics 
M\"{u}nster: Dynamics - Geometry - Structure} and by the DFG priority 
program SPP 2265 \textit{Random Geometric Systems}.}
\and
Tobias M\"uller\thanks{Bernoulli Institute, Groningen University, The Netherlands. E-mail: {\tt tobias.muller@rug.nl}.} 
}

\date{\today}

\maketitle

\begin{abstract} 
We study the typical cell of the Poisson-Voronoi tessellation.
We show that when divided by the $d$-th root of the intensity parameter $\lambda$ of the Poisson process times
the volume of the unit ball, the inradius, outradius, diameter and mean width of the typical cell converge in 
probability to the constants $1/2, 1, 2, 2$ respectively, as the dimension $d\to\infty$.
We also show that the width of the typical cell, when rescaled in the same way, is bounded between 
$2\sqrt{5}/(2+\sqrt{5})-o_d(1)$ and $3/2+o_d(1)$, with probability $1-o_d(1)$. 
These results in particular imply that, with probability $1-o_d(1)$, the Hausdorff distance between the 
typical cell and any ball is at least of the order of the diameter of the typical cell.

In addition, we show that for all $k$ with $d-k\to\infty$, with probability $1-o_d(1)$, all faces of dimension $k$ 
have a diameter that is of a much smaller 
order than the diameter, inradius, etc.,~of the full typical cell. 
The same is true for ``almost all'' faces of 
dimension $d-k$ with $k$ fixed. And, we show that the number of such faces is 
$\left( (k+1)^{(k+1)/2} / k^{k/2} \pm o_d(1) \right)^d$ with probability $1-o_d(1)$.
\end{abstract}

\section{Introduction and statement of results}

Throughout this paper, $\Zcal$ will denote a Poisson point process on $\eR^d$ of constant intensity $\lambda>0$.
The {\em Voronoi cell} of a point $z \in \Zcal$ is defined as 

$$ \Vcal(z) := \{ x \in \eR^d : \norm{x-z} \leq \norm{x-z'} \text{ for all $z'\in \Zcal$ }\}. $$

The Voronoi cells $\Vcal(z) : z \in \Zcal$ constitute a dissection of $\eR^d$, called the {\em Poisson-Voronoi} 
tessellation. A standard fact states that, almost surely, every Voronoi cell $\Vcal(z)$ is a convex polytope
with $z$ in its interior. (For a proof, see e.g.~\cite{SchneiderWeil}, Lemma 10.1.1 and the discussion 
immediately preceding Theorem 10.2.1.)
The Poisson-Voronoi tessellation is one of the central models
in stochastic geometry. It is studied in connection with many different applications and has 
a long history going back at least to the work of Meijering~\cite{Meijering} in the early fifties. 
For an overview, see the monographs~\cite{SchneiderWeil,StoyanKendallMecke87} and the references therein.

We will be considering the {\em typical cell}, which is the Voronoi 
cell of the  origin $\orig$ in the Voronoi tessellation for $\Zcal \cup \{\orig\}$, the Poisson point process with the origin 
added in. Throughout the paper we will denote the typical cell by $\Vtyp$.
Again, $\Vtyp$ is almost surely a polytope with the origin in its interior (see again~\cite{SchneiderWeil} for the proof).

The significance of $\Vtyp$ is that, as the name ``typical cell'' suggests, it describes the 
average behaviour of the cells of the Poisson-Voronoi tessellation.
For instance, if we take the fraction of all Poisson points $z \in \Zcal$ inside the ball $B(\orig,r)$ around the 
origin of radius $r$
for which $\Vcal(z)$ has precisely $k$ vertices then, as $r\to\infty$, this fraction converges
almost surely to $\Pee( \Vtyp \text{ has $k$ vertices} )$.
(Here and in the rest of the paper $B(x,r)$ denotes the $d$-dimensional open ball with center $x$ and radius $r$.)
More generally, for $h$ a translation invariant and appropriately measurable function from the space of 
polytopes
into $\eR_+$, we have

\begin{equation}\label{eq:ergodic} 
\lim_{r\to\infty} \frac{\sum_{z\in\Zcal\cap B(\orig,r)} h\left( \Vcal(z)\right)}{\left|\Zcal \cap B(\orig,r)\right|}
= \Ee h( \Vtyp ) \quad \text{ a.s. } 
\end{equation}

\noindent
(See e.g.~\cite{Calkanewpers} and the references therein.)

The Poisson-Voronoi tessellation is a classical subject in stochastic geometry and its typical cell has been studied 
quite extensively. The behaviour of the typical cell as the dimension grows is however still relatively unexplored.
Here we will study the behaviour of some parameters of $\Vtyp$ related to its shape and size, as the dimension grows.
Namely, we consider the inradius, outradius, diameter, width 
and mean width, denoted by $\inr(.)$, $\outr(.)$, $\diam(.)$, $\width(.)$, $\meanw(.)$, respectively.
(Detailed definitions can be found in Section~\ref{sec:prelim} below.)
We denote by $\vol(.)$ the $d$-dimensional volume (Lebesgue measure) and by $B := B(\orig,1)$ the 
$d$-dimensional unit ball.

\begin{theorem}\label{thm:main}
For any $\lambda = \lambda(d)>0$ we have 

$$ \frac{\outr(\Vtyp)}{\sqrt[d]{\lambda\cdot\vol(B)}} \xrightarrow[d\to\infty]{\Pee} 1,
\hspace{5ex} \frac{\inr(\Vtyp)}{\sqrt[d]{\lambda\cdot\vol(B)}} \xrightarrow[d\to\infty]{\Pee} \frac12, $$

$$ \frac{\diam(\Vtyp)}{\sqrt[d]{\lambda\cdot\vol(B)}}  \xrightarrow[d\to\infty]{\Pee} 2, \hspace{5ex} 
\frac{\meanw(\Vtyp)}{\sqrt[d]{\lambda\cdot\vol(B)}} \xrightarrow[d\to\infty]{\Pee} 2. $$

\end{theorem}

\noindent
To clarify, we emphasize that in the above theorem the intensity $\lambda$ is allowed to vary with the dimension.
Applying Stirling's approximation to the Gamma function in the expression~\eqref{eq:volball} for $\vol(B)$ below,
it is easily seen that in the case when $\lambda > 0$ is a fixed constant, the denominators
in Theorem~\ref{thm:main} can be replaced by $\sqrt{2\pi e/ d}$.

We would like to mention that while finalizing this paper, we learned that K.~Alishahi has shown in his PhD 
thesis~\cite{AlishahiThesis}
that the ratio $\outr(\Vtyp)/\inr(\Vtyp)$ tends to 2 in probability as $d\to\infty$, using different methods 
to the ones we employ. Alishahi's thesis is written in Persian and does not appear to 
be available online at the time of writing.

We also offer the following less precise result on the width of the typical cell:

\begin{proposition}\label{prop:width}
We have, for every fixed $\eps>0$ and $\lambda = \lambda(d) > 0$:

$$ \Pee\left( \frac{2\sqrt{5}}{2+\sqrt{5}}-\eps \leq \frac{\width(\Vtyp)}{\sqrt[d]{\lambda\cdot\vol(B)}} 
\leq \frac{3}{2}+\eps \right) = 1 - o_d(1). $$

\end{proposition}

To clarify, we emphasize that in the above result we choose $\eps>0$ fixed as the dimension 
increases towards infinity.
Note that $2\sqrt{5}/(2+\sqrt{5})\approx 1.0557 > 1$. So Proposition~\ref{prop:width} 
shows that the width is both bounded away by a multiplicative constant from twice the inradius (a trivial lower bound) 
and bounded away by a multiplicative constant from the diameter (a trivial upper bound).
We were not able to prove, but conjecture that $\width(\Vtyp)/\sqrt[d]{\lambda \cdot \vol(B)}$ converges to a constant
in probability.

\vspace{1ex}

A well-known folklore result states that $\Ee \vol(\Vtyp) = 1/\lambda$, in any dimension. (This can for instance be
seen from~\eqref{eq:ergodic} together with some relatively straightforward considerations.) 
Alishahi and Sharifitabar~\cite{AlishahiSharifitabar08} have shown that $\Var\left(\vol(\Vtyp)\right) = o_d(1)$ as 
the dimension $d\to\infty$ and the intensity $\lambda$ is kept constant, which 
in particular implies that the volume of $\Vtyp$ tends to $1/\lambda$ in probability as 
$d\to\infty$. In the light of~\eqref{eq:ergodic}, this can be informally paraphrased as: ``in high dimension, almost all
Voronoi cells have volume arbitrarily close to $1/\lambda$''.
The result of Alishahi and Sharifitabar was later extended by Yao~\cite{Yao}, who showed that the variance of the 
volume of the intersection of $\Vtyp$ with any measurable set is bounded by the variance of the volume of $\Vtyp$.
Additional results in~\cite{AlishahiSharifitabar08} imply that if $r = \sqrt[d]{1/\lambda \cdot \vol(B)}$ is such that 
$\vol(B(\orig,r)) = 1/\lambda$ then almost all mass of $\Vtyp$ is contained in $B(\orig,(1+\eps)r)$ and 
almost all mass of $B(\orig,(1-\eps)r)$ is contained in $\Vtyp$.
By this we mean that both the ratios $\vol(\Vtyp \cap B(\orig,(1-\eps) r)) / \vol(B(\orig,(1-\eps)r))$ and 
$\vol(\Vtyp \cap B(\orig,(1+\eps) r)) / \vol(\Vtyp)$ tend to one in probability as 
the dimension $d\to\infty$.
This can be interpreted as credence for the idea that the typical cell is somehow ``ball like'', at least as far as the 
volume is concerned.
On a similar note, H\"ormann et al.~\cite{HoermannEtal} investigated the asymptotics of the expected number of 
$k$-faces of $\Vtyp$ as the dimension goes to infinity.
Based on their findings, they mention (page 13, paragraph following Theorem 3.20) ``... we roughly speaking see that 
the typical Poisson-Voronoi cells are approximately spherical in the mean ...''.

In contrast, our results on diameter, mean width and outradius all seem to support the idea the typical cell is somehow close to 
a ball of the same volume, while the (proof of the) results on the inradius and width 
suggest the typical cell behaves rather differently from a ball of the same volume. 
Our results for instance imply that the Hausdorff distance between $\Vtyp$ and any ball is large (i.e.~at least 
of the same order as the diameter of $\Vtyp$). 
For completeness we spell out the short argument demonstrating this in Appendix~\ref{sec:Hausdorff}.

During the course of the proofs of the above results, we will derive and heavily rely on 
the following observation that is of independent interest: 
with probability $1-o_d(1)$, {\em all} vertices of $\Vtyp$ have norm 
close to $\sqrt[d]{\lambda\cdot\vol(B)}$.
This generalizes as follows. Let $\Fcal_k(P)$ denote the set of $k$-faces of the polytope $P$.
The union $\bigcup \Fcal_k(P)$ of all $k$-faces is sometimes also called the $k$-skeleton of $P$.

\begin{proposition}\label{prop:radiussmallface}
Let $\eps>0$ be fixed and $k=k(d)$ be such that $d-k\to \infty$.
For any $\lambda=\lambda(d)>0$ we have:

$$ \Pee\left( \bigcup \Fcal_k(\Vtyp) \subseteq B(\orig,(1+\eps)r)
\setminus B(\orig,(1-\eps)r) \right) 
\xrightarrow[d\to\infty]{} 1, $$

\noindent
where $r = r(d,\lambda) := \sqrt[d]{\lambda\cdot\vol(B)}$.
\end{proposition}

This last proposition tells us that, with probability $1-o_d(1)$, all faces of non-constant co-dimension 
are contained in an annulus around $\orig$ of width $o\left( \sqrt[d]{\lambda \vol(B)} \right)$ and 
(inner) radius $(1+o_d(1)) \cdot \sqrt[d]{\lambda \vol(B)}$.
From this it can be seen that, with probability $1-o_d(1)$, every face of non-constant co-dimension
is ``microscopic'' (i.e.~of diameter $o\left( \sqrt[d]{\lambda \vol(B)} \right)$).

\begin{corollary}\label{cor:diamsmallface}
 Let $k=k(d)$ be such that $d-k\to \infty$.
For any $\lambda=\lambda(d)>0$ we have:

$$ \frac{\max_{F \in \Fcal_k(\Vtyp)} \diam(F)}{\sqrt[d]{\lambda\cdot\vol(B)}} 
\xrightarrow[d\to\infty]{\Pee} 0.
$$ 
\end{corollary}

\noindent
(For completeness, we spell out the short derivation of Corollary~\ref{cor:diamsmallface} from 
Proposition~\ref{prop:radiussmallface} in Section~\ref{sec:smallface}.)

Theorem~\ref{thm:main} tells us that Proposition~\ref{prop:radiussmallface} 
does not extend to facets, i.e.~faces of co-dimension one. (If all facets of a polytope $P$ with $\orig \in P$ lie in the annulus from Proposition~\ref{prop:radiussmallface} then 
$\outr(P)/\inr(P)$ must be close to one.)
What is more, Corollary~\ref{cor:diamsmallface} does not extend to facets. 
As part of the proof of Proposition~\ref{prop:width} we will exhibit the existence (with probability $1-o_d(1)$)
of facets that have distance roughly $1/2 \cdot \sqrt[d]{\lambda\cdot\vol(B)}$ to the origin.
Since, with probability $1-o_d(1)$, all vertices have distance roughly $\sqrt[d]{\lambda\cdot\vol(B)}$ to the origin,
some pair of vertices 
on such a face will have distance $\Omega\left( \sqrt[d]{\lambda\cdot\vol(B)} \right)$.
Although we will not prove this explicitly here, similar remarks apply to faces of any constant co-dimension.

However, as we will show, the vast majority of the faces of constant co-dimension have 
microscopic diameter. We let $f_k(P) := |\Fcal_k(P)|$ denote the number of $k$-faces of the polytope $P$.

\begin{theorem}\label{thm:diambigface}
For every fixed $k \in \eN$ and every $\lambda=\lambda(d)>0$, we have 

$$ \frac{{\displaystyle \frac{1}{f_{d-k}(\Vtyp)}}\cdot\sum_{F \in \Fcal_{d-k}(\Vtyp)} \diam(F)}{\sqrt[d]{\lambda\cdot\vol(B)}} 
\xrightarrow[d\to\infty]{\Pee} 0.
$$ 
\end{theorem}

\noindent
(For clarity, we remark that this last theorem implies that all but a negligible proportion of the faces of co-dimension $k$ 
have diameter $o\left( \sqrt[d]{\lambda \vol(B)} \right)$.)
We point out that, since each face contains a vertex, Theorem~\ref{thm:diambigface} and Proposition~\ref{prop:radiussmallface} 
together imply that, with probability 
$1-o_d(1)$, all but a negligible proportion of faces of co-dimension $k$ are contained 
in an annulus around $\orig$ of width $o\left( \sqrt[d]{\lambda \vol(B)} \right)$ and 
(inner) radius $(1+o_d(1)) \cdot \sqrt[d]{\lambda \vol(B)}$.

Some observations that we shall make in the course of the proof are of independent interest.
The first of these estimates the number of faces of fixed co-dimension $k$.
The (expected) number of faces of $\Vtyp$ is a theme that has been considered a fair bit 
in the literature (see e.g.~\cite{HoermannEtal,ZakharAdvMath}), but as far as we know the following result is new.

\begin{theorem}\label{thm:fass}
For every fixed $k \in \eN$ and every $\lambda=\lambda(d)>0$, we have 

$$ \sqrt[d]{f_{d-k}\left(\Vtyp\right)} \xrightarrow[d\to\infty]{\Pee} 
\frac{(k+1)^{(k+1)/2}}{k^{k/2}}.
$$

\end{theorem}

In a forthcoming article we will derive more precise asymptotics for $f_{d-k}\left(\Vtyp\right)$.
The constant in the right hand side of this last theorem is equals $k!$ times the $k$-dimensional volume of 
a regular $k$-simplex inscribed in $S^{k-1}$.
Each edge of such a simplex has length $\ell_k := \left(\frac{2(k+1)}{k}\right)^{1/2}$.
For $\eps>0$, we will say that $z_1,\dots,z_k \in \eR^d$ is an {\em $\eps$-near regular simplex} if 
$\ell_k - \eps < \norm{z_i-z_j} < \ell_k+\eps$ for all $0\leq i < j \leq k$ where $z_0 := \underline{0}$
denotes the origin.
We say that $z_1, \dots, z_k \in \Zcal$ {\em define a $(d-k)$-face of $\Vtyp$} if 
there is a face $F \in \Fcal_{d-k}(\Vtyp)$ such that 

$$ F = \Vtyp \cap \bigcap_{i=1}^k \left\{ x \in \eR^d : \norm{x} = \norm{x-z_i} \right\}. $$

Let $M_{k,\eps}$ denote the number of $(d-k)$-faces of $\Vtyp$ that 
are defined by a $k$-set $\{z_1,\dots,z_k\} \subseteq \Zcal$ 
that forms an $\eps$-near regular simplex.
The following observation tells us that ``almost all'' faces of constant co-dimension 
are defined by near-regular simplices.

\begin{proposition}\label{prop:Nkeps}
For $k \in \eN$ and $\eps>0$ fixed and any $\lambda=\lambda(d)>0$ we have that 

$$ \frac{M_{k,\eps}}{f_{d-k}(\Vtyp)} \xrightarrow[d\to\infty]{\Pee} 1. $$

\end{proposition}

\section{Notation and preliminaries\label{sec:prelim}}

We will use $B(x,r) := \{ y \in\eR^d : \norm{y-x} < r\}$ to denote the open $d$-dimensional ball with center $x$ and
radius $r$.
We will use $\kappa_d$ to denote the $d$-dimensional volume of the $d$-dimensional unit ball. That is,

\begin{equation}\label{eq:volball} 
\kappa_d := \vol( B ) = \frac{\pi^{d/2}}{\Gamma(d/2+1)},  
\end{equation}

\noindent
where the stated equality is a classical result (see e.g.~\cite{Schilling}, Corollary 15.15, for a proof).

We briefly recall the definitions of the main parameters we'll be studying in the present paper.
The {\em diameter} of a set $A \subseteq \eR^d$ is

$$ \diam(A) := \sup_{a,b\in A} \norm{a-b}. $$

\noindent
The {\em inradius} and {\em outradius} are defined respectively by

$$ \begin{array}{l} 
\inr(A) := \sup \{ r > 0 : \text{$\exists p \in \eR^d$ such that $B(p,r) \subseteq A$}\}, \\
\outr(A) := \inf \{ r > 0 : \text{$\exists p \in \eR^d$ such that $B(p,r) \supseteq A$}\}.    
\end{array} $$

\noindent
We remark that (in the case when $A$ is the typical cell of 
a Poisson-Voronoi tessellation) when defining the inradius, respectively outradius, 
some authors (e.g.~\cite{CalkaChevanier}) take the $\sup$, respectively $\inf$, 
of the radius of all balls centered on the origin $\orig$ that are contained in $A$, respectively contain $A$.

We denote by $S^{d-1} \subseteq \eR^d$ the $(d-1)$-dimensional unit sphere (in ambient $d$-dimensional space).
The {\em width in the direction $u \in S^{d-1}$} is defined by

$$ w(u,A) := \sup_{a\in A} u^t a - \inf_{a\in A} u^t a. $$

\noindent
Here and in the rest of the paper $v^t w$ denotes the inner product of the vectors $v,w \in \eR^d$.
The (ordinary) {\em width} is 

$$ \width(A) := \inf_{u\in S^{d-1}} w(u,A). $$

\noindent
The {\em mean width}, as its name suggests, is the average of $w(u,A)$ over all $u \in S^{d-1}$. 
One of several equivalent ways to define this more formally is by setting

$$ \meanw(A) := \Ee w(U,A), $$

\noindent
where $U$ is point chosen uniformly at random on $S^{d-1}$.

For $x \in \eR^d \setminus \{\orig\}$ we will denote by 

\begin{equation}\label{eq:Hxdef} 
H_x := \{ y \in \eR^d : x^t y \leq x^tx/2  \},, 
\end{equation}

\noindent
for the set of all points that are at least as close to the origin $\orig$ as they are to $x$.

We note that the typical cell of the Poisson-Voronoi tessellation can be written as
\begin{equation}\label{eq:altCdef} 
\Vtyp = \bigcap_{z\in\Zcal} H_z, 
\end{equation}

\noindent
where $\Zcal$ denotes the Poisson point process that generates the typical cell $\Vtyp$.

We'll make use of the following bound from the literature. 
The following is part (b) of Lemma 2.1 in~\cite{BriedenEtal}, rephrased in terms of angles rather than 
spherical caps.
The notation $\angle abc$ denotes the angle between the vectors $a-b$ and $c-b$.

\begin{lemma}[\cite{BriedenEtal}]\label{lem:X}

 Let $U$ be chosen uniformly at random on the unit sphere $S^{d-1}$ in $\eR^d$.
 For any $v \in \eR^d \setminus \{\orig\}$ and $0< \alpha < \arccos(\sqrt{2/d})$ we have
 
 $$ \frac{\sin^{d-1}\alpha}{6\cdot\cos\alpha\cdot\sqrt{d}} \leq
 \Pee( \angle U{\orig}v < \alpha ) \leq \frac{\sin^{d-1}\alpha}{2\cdot \cos\alpha\cdot\sqrt{d}}. $$
 
 \end{lemma}

We'll make use of the following incarnation of the Chernoff bound.
A proof can for instance be found in~\cite{Penroseboek} (Lemma 1.2). 
We use $\Po(\mu)$ to denote the Poisson distribution with mean $\mu$.

\begin{lemma}[Chernoff bound]\label{lem:chernoff}
Let $X \isd \Po(\mu)$ and $x \geq \mu$. Then 

$$ \Pee( X \geq x ) \leq e^{-\mu H(x/\mu)}, $$
 
\noindent
where $H(z) := z\ln z-z+1$.
\end{lemma}

In the proofs below, we will repeatedly make use of the dissection $\Qcal_\delta$ of $\eR^d$ into axis parallel cubes of side length 
$\delta/\sqrt{d}$ (and hence diameter $\delta$) in the obvious way, where $\delta>0$ will be a constant chosen 
independently of the dimension.
In more detail:

$$ \Qcal_\delta := \left\{ 
\left[\frac{i_1\delta}{\sqrt{d}}, \frac{(i_1+1)\delta}{\sqrt{d}}\right) \times \cdots \times
\left[\frac{i_d\delta}{\sqrt{d}}, \frac{(i_d+1)\delta}{\sqrt{d}}\right) :
i_1,\dots, i_d \in \Zed \right\}. $$

\noindent
We point out that if 

$$ \Qcal_{\delta,R} := \left\{ q \in \Qcal_\delta : q \subseteq B(\orig,R) \right\}, $$

\noindent 
and 

\begin{equation}\label{eq:Ndef}
 N_{\delta,R} := \left| \Qcal_{\delta,R}\right|, 
\end{equation}

\noindent
is the number of cubes in the dissection that are contained in the ball of radius $R$ around the origin, then 

\begin{equation}\label{eq:Nub} 
N_{\delta,R} \leq \frac{\vol(B(\orig,R))}{(\delta/\sqrt{d})^d}
 = \frac{\kappa_d R^d}{(\delta/\sqrt{d})^d} = (R/\delta)^d \cdot \exp[  O(d) ],
\end{equation}

\noindent
where the asymptotics is for $d \to \infty$ and the term $O(d)$ 
depends only on $d$ and not $R$ or $\delta$. (Here we've used that 
$\kappa_d = d^{-d/2} \cdot \exp\left[  O(d) \right]$, as can for instance be seen
from Stirling's approximation to the Gamma function and the exact expression~\eqref{eq:volball} for $\kappa_d$ above.)

\section{Proofs}

We recall that a dilation of an intensity $\lambda$ Poisson point process on $\eR^d$ by a factor of
$\rho>0$ yields a Poisson process of intensity $\rho^{-d}\lambda$ whose typical cell is 
just the typical cell of the original process rescaled by $\rho$.
In particular, for the proofs of our results we can take any value of $\lambda$ 
we like (and the result will follow for all choices $\lambda$).

We find it convenient to take 

$$\lambda := 1/\kappa_d.$$ 

\noindent
(This way the numerator in our main results equals one; and also 
the expected number of Poisson points in a ball of radius $s$ is simply $s^d$).
We will be using this choice of $\lambda$ throughout the remainder of the paper, without 
stating it explicitly every time.
We will denote by $\Zcal \subseteq \eR^d$ the Poisson point process of intensity $\lambda = 1/\kappa_d$ on $\eR^d$.

\subsection{All vertices have norm approximately one}

\begin{lemma}\label{lem:vertlb}
For every fixed $\eps>0$, with probability $1-o_d(1)$, all vertices of $\Vtyp$ have norm $>1-\eps$.
\end{lemma}

\begin{proof}
We use the elementary observation that if $p \in \eR^d$ is a vertex of $\Vtyp$ then 
 there are $d$ points $z_1,\dots,z_d \in \Zcal$ of the Poisson process such that 
 $z_1,\dots,z_d \in \partial B(p, \norm{p})$ and in addition $\Zcal \cap B(p, \norm{p}) = \emptyset$.
(This because a vertex of $\Vtyp$ must lie on the common boundary of $d+1$ Voronoi cells, one of which must be $\Vtyp$
itself. This translates to $p$ being  equidistant to $\orig$ and $d$ points of the Poisson point process $\Zcal$, 
and in addition no point of $\Zcal$ can be closer to $p$ than the common distance -- which must be $\norm{p}$.)

We let $\delta=\delta(\eps) < \eps/1000$ be a small constant.
Any vertex $v$ of $\Vtyp$ with norm $\leq 1-\eps$ is contained in one of the cubes of $\Qcal_\delta$ as defined above.
That cube itself is completely contained in 
$B(\orig,1-\eps+\delta)\subseteq B$. Let $c$ be one of the corners of the cube containing $v$.
By the elementary observation above, $B(c,1-\eps+2\delta)$ contains at least $d$ points of the Poisson process.
Hence, if $E$ denotes the event that there exists a vertex of $\Vtyp$ with norm $\leq 1-\eps$ then by the union bound: 

$$ \Pee( E ) \leq N_{\delta,1} \cdot \Pee\left( \Po( (1-\eps+2\delta)^d ) \geq d \right), $$

\noindent
(The first term is an upper bound on the number of cubes considered and the second term is the probability 
the ball around a corner of a given cube has $\geq d$ Poisson points in it.)

By~\eqref{eq:Nub} we have $N_{\delta,1} = e^{O(d)}$.
On the other hand, using the version of the Chernoff bound we presented as Lemma~\ref{lem:chernoff} above, 
we see that

$$ \Pee\left( \Po( (1-\eps+2\delta)^d ) \geq d \right) \leq 
\exp\left[ - (1-\eps+2\delta)^d H\left( \frac{d}{(1-\eps+2\delta)^d}\right) \right]
= e^{-\Omega( d^2 )}. $$

\noindent
Combining these bounds, we see that $\Pee(E) = o_d(1)$, as claimed in the lemma statement.
\end{proof}

\begin{lemma}\label{lem:vertub}
For every fixed $\eps>0$, with probability $1-o_d(1)$, every vertex of $\Vtyp$ has norm $<1+\eps$. 
\end{lemma}

\begin{proof}
 As in the proof of the previous lemma, we use that if $p \in \eR^d$ is a vertex of $\Vtyp$ then 
 there are $d$ points $z_1,\dots,z_d \in \Zcal$ of the Poisson process such that 
 $z_1,\dots,z_d \in \partial B(p, \norm{p})$ and in addition $\Zcal \cap B(p, \norm{p}) = \emptyset$.
 
 We fix some $\delta=\delta(\eps) \in (0,\min(\eps/2,1))$.
 If $p \in B(\orig,2) \setminus B(\orig,1+\eps)$ is a vertex of $\Vtyp$ then 
 there is some cube $q \in \Qcal_\delta$ satisfying that $q \subseteq B(\orig,3)$ and 
 $B(c,1+\eps-\delta) \cap \Zcal = \emptyset$ for $c$ any corner of $q$.
 Similarly, if for some $n\geq 2$ a point $p \in B(\orig,n+1) \setminus B(\orig,n)$ is 
 a vertex of $\Vtyp$ then 
 there is some cube $q \in \Qcal_\delta$ satisfying that $q \subseteq B(\orig,n+2)$ and 
 $B(c,n-\delta) \cap \Zcal = \emptyset$ for $c$ any corner of $q$.
 
 It follows that if $E$ is the event that $\Vtyp$ has at least one vertex outside $B(\orig,1+\eps)$ then 
 
 $$ \begin{array}{rcl} \Pee(E) & \leq & N_{\delta,3} \cdot \Pee( \Po( (1+\eps-\delta)^d ) = 0 ) 
 + \sum_{n\geq 2} N_{\delta,n+2} \cdot \Pee( \Po( (n-\delta)^d ) = 0 ) \\
& = & e^{O(d) - (1+\eps-\delta)^d} 
+ \exp[ O(d) ] \cdot \sum_{n\geq 2} (n+2)^d \cdot  e^{-(n-\delta)^d} \\
& = & o_d(1),
\end{array} $$

\noindent
using~\eqref{eq:Nub} above.
\end{proof}

\subsection{The outradius, diameter and mean width}

Since $\Vtyp$ is the convex hull of its vertices, Lemma~\ref{lem:vertub} implies that 
$\Vtyp \subseteq B(\orig,1+\eps)$ (for any fixed $\eps>0$, with probability $1-o_d(1)$). 
In particular:

\begin{corollary}\label{cor:twaalf}
For every $\eps>0$, with probability $1-o_d(1)$, we have $\meanw(\Vtyp)\leq \diam(\Vtyp) < 2+\eps$ and $\outr(\Vtyp) < 1+\eps$.
\end{corollary}

We need to prove matching lower bounds.
The following observation does the trick for the diameter and outradius.

\begin{lemma}\label{lem:2pts}
For every $\eps>0$, with probability $1-o_d(1)$, we have 
$(1-\eps,0,\dots,0) \in \Vtyp$ and $(\eps-1,0,\dots,0) \in \Vtyp$.
\end{lemma}

\begin{proof}
By symmetry and the union bound, we just need to show that 
$\Pee( (1-\eps,0,\dots,0) \not\in \Vtyp ) = o_d(1)$.
We simply note that 

$$ \begin{array}{rcl} 
\Pee\left( (1-\eps,0,\dots,0) \in \Vtyp\right) & = &  
\Pee( \Zcal \cap B( (1-\eps,0,\dots,0), 1-\eps) = \emptyset ) \\
& = & \Pee( \Po( (1-\eps)^d ) = 0 ) \\
& = & e^{-(1-\eps)^d } \\
& = & 1-o_d(1). 
\end{array} $$

\end{proof}

\begin{corollary}
For every $\eps>0$, with probability $1-o_d(1)$, we have $\diam(\Vtyp) > 2-\eps$ and $\outr(\Vtyp) > 1-\eps$.
\end{corollary}

For the mean width the lower bound also follows quite easily.

\begin{lemma}\label{lem:meanwlb}
 For every $\eps>0$, with probability $1-o_d(1)$, we have 
 $\meanw(\Vtyp) > 2-\eps$.
\end{lemma}

\begin{proof}
We pick $\delta=\delta(\eps)>0$ a small constant to be determined.
We define

$$ \Xcal := \{ u \in S^{d-1} : w(u,\Vtyp) < 2-\delta \}, \quad X := \sigma(\Xcal)/\sigma(S^{d-1}), $$

\noindent
where $\sigma$ denotes the Haar measure on $S^{d-1}$ (the ``$(d-1)$-dimensional area'').
Put differently, $X$ is the fraction of $S^{d-1}$ that is covered by directions in 
which the width is $<2-\delta$.
We have 

$$ \meanw(\Vtyp) \geq (1-X) \cdot (2-\delta). $$

\noindent
Next, we observe that 

$$ \begin{array}{rcl} \Ee X 
& = & \displaystyle \Ee\left( \frac{1}{\sigma(S^{d-1})} \int_{S^{d-1}} 1_{\{w(u,\Vtyp)<2-\delta\}}\sigma({\dd}u) \right) 
= \frac{1}{\sigma(S^{d-1})} \int_{S^{d-1}} \Ee\left( 1_{\{w(u,\Vtyp)<2-\delta\}} \right) \sigma({\dd}u) \\[2ex]
& = & \displaystyle \frac{1}{\sigma(S^{d-1})} \int_{S^{d-1}} \Pee( w(u,\Vtyp)<2-\delta ) \sigma({\dd}u) 
= \displaystyle \frac{1}{\sigma(S^{d-1})} \int_{S^{d-1}} \Pee( w(e_1,\Vtyp)<2-\delta ) \sigma({\dd}u) \\[2ex]
& = & \displaystyle \Pee( w(e_1,\Vtyp)<2-\delta ).
\end{array} $$

\noindent
where $e_1 = (1,0,\dots,0)$ denotes the first standard basis vector, 
we used Fubini for non-negative integrands in the second equality and symmetry in the fourth.
Applying Lemma~\ref{lem:2pts}:

$$ 
\Ee X = \Pee( w(e_1,\Vtyp) < 2-\delta )  \leq   
1-\Pee( (1-\delta/2,0,\dots,0), (\delta/2-1,0,\dots,0) \in \Vtyp ) = o_d(1). 
$$

\noindent
Markov's inequality thus gives

$$ \Pee( X > \delta ) \leq \Ee X / \delta = o_d(1). $$

\noindent
It follows that, with probability $1-o_d(1)$, $\meanw(\Vtyp) > (2-\delta)\cdot(1-\delta) > 2-\eps$, where 
the last inequality holds having chosen $\delta$ appropriately.
\end{proof}

\subsection{The inradius}

A lower bound on the inradius of $\Vtyp$ is given by the following lemma.

\begin{lemma}
For every $\eps>0$, with probability $1-o_d(1)$, we have $B(\orig,\frac12-\eps) \subseteq \Vtyp$.
\end{lemma}

\begin{proof}
We simply note that if $\Zcal \cap B(\orig,1-2\eps) = \emptyset$ then $B(\orig,1/2-\eps) \subseteq \Vtyp$, and
hence 

$$ \Pee\left( B(\orig,1/2-\eps\right) \subseteq \Vtyp ) \geq 
\Pee( \Po( (1-2\eps)^d ) = 0 ) = e^{-(1-2\eps)^d} = 1-o_d(1). $$

\end{proof}

To derive a matching upper bound we need to show that {\em no} ball of radius $\frac12 + \eps$ is contained in $\Vtyp$
(with probability $1-o_d(1)$). 
Even if $B(\orig,\frac12+\eps) \not\subseteq \Vtyp$ there are in principle still infinitely
many balls of the same radius but with a different center that need to be excluded.
In order to deal with this issue, we will make use of the following observation.

\begin{lemma}
For every $p \in \eR^d$ and $r>0$ we have

$$ \Pee( B(p,r) \subseteq \Vtyp) \leq \sqrt{\Pee( B(\orig,r) \subseteq \Vtyp )}. $$

\end{lemma}

\begin{proof}
By symmetry, we can take $p$ on the positive $x_1$-axis without loss of generality.
We set 

$$ A := B(\orig,2r), \quad A^+ := A \cap \{ x \in \eR^d : x_1 > 0 \}. $$

\noindent
We note that $A$ can be written alternatively as $A = \{ x \in \eR^d : B(\orig,r) \not\subseteq H_x \}$, 
where $H_x$ is as defined by~\eqref{eq:Hxdef}.
Appealing to the observation~\eqref{eq:altCdef}, we have

$$ \Pee( B(\orig,r) \subseteq \Vtyp ) = \Pee( \Zcal \cap A = \emptyset ). $$

We claim that if $\Zcal \cap A^+ \neq \emptyset$ then $B(p,r) \not\subseteq \Vtyp$.
To see this, let $x \in A^+$ be arbitrary.
Then there exists a $y \in B(\orig,r) \setminus H_x$.
That is, $y \in B(\orig,r)$ and $x^t y > x^tx / 2$.
But then $w := y + p \in B(p,r)$ and 

$$ x^t w = x^t y + x^t p > x^t y > x^t x / 2, $$

\noindent
where the first inequality holds by choice of $p$ and $A^+$ ($p$ is on the positive $x_1$-axis and $x$'s first coordinate is 
positive). This establishes that $B(p,r) \setminus H_x \neq \emptyset$ for all $x \in A^+$.
In other words, we've proven the claim that $\{\Zcal \cap A^+ \neq \emptyset\} \subseteq  \{B(p,r) \not\subseteq \Vtyp\}$.

This gives

$$ \Pee( B(p,r) \subseteq \Vtyp ) \leq \Pee( \Zcal \cap A^+ = \emptyset )
= e^{-\lambda\cdot\vol(A^+)}
= e^{-\lambda\cdot\vol(A)/2} 
= \sqrt{\Pee( \Zcal \cap A = \emptyset )}. $$

\noindent
The lemma follows.
\end{proof}

\begin{lemma}
 For every $\eps>0$, with probability $1-o_d(1)$, we have $\inr(\Vtyp) \leq \frac12 + \eps$.
\end{lemma}

\begin{proof}
Again we let $\delta=\delta(\eps)>0$ be an appropriately chosen small constant.
As we've already established that $\Vtyp \subseteq B(\orig,1+\eps)$, with probability $1-o_d(1)$,
it suffices to show that there is no $p \in B(\orig,1/2)$ such that $B(p,1/2+\eps) \subseteq \Vtyp$.
Each such $p$ lies inside some cube $q \in \Qcal_{\delta,1}$, and any corner $c$ of $q$ 
must satisfy that $B(c,1/2+\eps-\delta) \subseteq \Vtyp$.

Applying the previous lemma we find

$$ \begin{array}{rcl} 
\Pee( \inr(\Vtyp) > 1/2+\eps ) 
& \leq &  
\Pee\left( \Vtyp \not\subseteq  B(\orig,1+\eps) \right) 
+ N_{\delta,1} \cdot \sqrt{\Pee( B(\orig,1/2+\eps-\delta) \subseteq \Vtyp )} \\
& = & o_d(1) + e^{O(d)} \cdot \sqrt{\Pee( \Zcal \cap B(\orig,1+2\eps-2\delta) = \emptyset )} \\
& = & o_d(1) + e^{O(d)} \cdot e^{-(1+2\eps-2\delta)^d/2 } \\
& = & o_d(1), \end{array} $$

\noindent
having chosen $0<\delta<\eps$ appropriately.
\end{proof}

\subsection{Upper bound for the width}

\begin{lemma}
 For every $\eps>0$, with probability $1-o_d(1)$, we have
 $\width(\Vtyp) < 3/2 + \eps$.
\end{lemma}

\begin{proof}
As noted before, it follows from Lemma~\ref{lem:vertub} that $\Vtyp \subseteq B(\orig,1+\eps)$ with 
probability $1-o_d(1)$.
Therefore we also have, with probability $1-o_d(1)$, that 

$$ \Vtyp \subseteq B(\orig,1+\eps) \cap H_z, $$

\noindent
for all $z \in \Zcal$ -- where $H_z$ is as defined by~\eqref{eq:Hxdef} above and we use~\eqref{eq:altCdef}. 
This gives that, with probability $1-o_d(1)$:

$$ \begin{array}{rcl} 
\width(\Vtyp) 
& \leq & \displaystyle 
\inf_{z\in\Zcal}  \width( B(\orig,1+\eps) \cap H_z ) \\[2ex]
& = & \displaystyle
\inf_{z\in\Zcal}\inf_{u\in S^{d-1}} w( u, B(\orig,1+\eps) \cap H_z ) \\[2ex]
& \leq & \displaystyle 
\inf_{z\in\Zcal} w\left(\frac{z}{\norm{z}}, B(\orig,1+\eps) \cap H_z \right) \\[2ex]
& \leq & \displaystyle 
1+\eps + \inf_{z\in\Zcal} \norm{z}/2. 
\end{array} $$

\noindent 
Since 

$$ \Pee( \Zcal \cap B(\orig,1+\eps) \neq \emptyset ) = 1 - e^{-(1+\eps)^d}  = 1 - o_d(1), $$

\noindent
it follows that, with probability $1-o_d(1)$, we have $\inf_{z\in\Zcal} \norm{z} < 1+\eps$.
In other words, with probability $1-o_d(1)$, $\width(\Vtyp) \leq (3/2)\cdot (1+\eps)$.
Adjusting the value of $\eps$, the result follows.
\end{proof}

\subsection{Lower bound on the width}

Recall that the {\em polar} of a set $A \subseteq \eR^d$ is defined by 

$$ A^{\circ} := \left\{ y \in \eR^d : a^t y \leq 1 \text{ for all } a \in A \right\}. $$

\noindent
(For background and an overview of properties of the polar set, see e.g.~\cite{Bronsted}.)
We will find it convenient to switch attention to the polar $\Vtyp^{\circ}$ of $\Vtyp$.
Let $f : \eR^d \to \eR^d$ be the map defined by $x \mapsto (2/\norm{x}^2) x$ for $x\neq\orig$ and
$\orig\mapsto\orig$.
We let the point set $\Ycal \subseteq \eR^d$ be defined by

$$ \Ycal := f[\Zcal], $$

\noindent
where $\Zcal$ as usual is the Poisson point process of constant intensity $\lambda=1/\kappa_d$ that we
used to define the typical cell $\Vtyp$.
By the mapping theorem (see e.g.~\cite{Kingman}, page 18) $\Ycal$ is also a Poisson point 
process on $\eR^d\setminus\{\orig\}$. 
Its intensity 
function can easily be worked out to be $\text{const} \cdot \norm{x}^{-2d}$, but we shall not be needing that.
Convex hulls of Poisson processes with power-law intensity and their duals were studied 
in~\cite{ZakharAdvMath, ZakharPAMS, KMTT, KTZ}.

We will rely on the following well-known observation.
For a proof,  see for instance the footnote on page 17 of~\cite{ZakharAdvMath}.

\begin{proposition}
Almost surely, $\Vtyp^{\circ} = \conv(\Ycal)$.
\end{proposition}

\noindent
Although we shall not need this fact below, let us remark that $\orig$ belongs to the interior 
of $\conv (\mathcal Y)$ almost surely (see Corollary~4.2 on page 1040 of~\cite{KMTT}).

We define for $u \in S^{d-1}$ and $A \subseteq  \eR^d$:

$$ \varphi(u,A) := \sup\{ u^t a : a \in A \}, 
\quad \psi(u,A) := \sup\{ \lambda \geq 0 : \lambda u \in A \}. $$

\noindent
We remark that we can write 

$$ w(u,A) = \varphi(u,A) + \varphi(-u,A). $$

Another well-known, elementary observation is the following.
It for instance occurs as equation (14.42) in~\cite{SchneiderWeil}.

\begin{lemma}\label{lem:phipsi}
For $A \subseteq \eR^d$ compact and convex with $\orig \in \inter A$ and 
 all $u \in S^{d-1}$ we have 

$$ \varphi(u,A) = \frac{1}{\psi(u,A^\circ)}. $$

\end{lemma}

\begin{lemma}\label{lem:angleswhp}
Let $0<r<2$ be fixed. 
With probability $1-o_d(1)$, any two distinct $Y_1, Y_2 \in \Ycal$ with $\norm{Y_1}>r,\norm{Y_2} > r$
satisfy 

$$\alpha < \angle Y_1{\orig}Y_2 < \pi-\alpha,$$

\noindent
where $\alpha := \arcsin\left(r^2/4\right) < \pi/2$.
\end{lemma}

\begin{proof} Let us write 

$$ \begin{array}{rcl} X 
& := & \displaystyle 
\left|\left\{ (Y_1,Y_2) \in \Ycal^2 : \begin{array}{l} \norm{Y_1},\norm{Y_2} > r \text{ and } Y_1\neq Y_2 \text{ and } \\
\angle Y_1{\orig}Y_2 \leq \alpha \text{ or } \angle Y_1{\orig}Y_2 \geq \pi-\alpha\end{array} \right\}\right| \\[2ex]
& = & \displaystyle 
\left|\left\{ (Z_1, Z_2) \in \Zcal^2 : \begin{array}{l} \norm{Z_1},\norm{Z_2} < 2/r \text{ and } Z_1\neq Z_2 \text{ and } \\
\angle Z_1{\orig}Z_2 \leq \alpha \text{ or } \angle Z_1{\orig}Z_2 \geq \pi-\alpha\end{array} \right\}\right|. 
\end{array} $$

Using the Mecke formula (see e.g.~\cite{SchneiderWeil}, Corollary 3.2.3), we find

$$ \begin{array}{rcl} \Ee X & = & \displaystyle 
\lambda^2 \int_{\eR^d}\int_{\eR^d} 1_{\left\{\begin{array}{l} \norm{z_1},\norm{z_2} < 2/r \text{ and }  \\
\angle z_1{\orig}z_2 \leq \alpha \text{ or } \angle z_1{\orig}z_2 \geq \pi-\alpha\end{array} \right\}}
\dd z_1 \dd z_2 \\[5ex]
& = & \displaystyle \lambda^2 \cdot \vol( B(\orig,2/r) )^2 \cdot \Pee( \angle U_1\orig U_2 \leq \alpha \text{ or }
\angle U_1\orig U_2 \geq \pi-\alpha ) \\[2ex]
& = & 2 \cdot (4/r^2)^{d} \cdot \Pee( \angle U_1\orig v \leq \alpha ) \\[2ex]
& = & o\left( (4/r^2)^{d} \cdot \sin^d \alpha \right) \\[2ex]
& = & o_d(1),   
 \end{array} $$
 
\noindent
where $U_1,U_2$ denote points chosen uniformly at random on $S^{d-1}$ and $v \neq \orig$ is an arbitrary 
but fixed point in $\eR^d$. 
(The second line follows by symmetry considerations; the third line uses the choice of $\lambda$ and 
more symmetry considerations, and; 
we applied Lemma~\ref{lem:X} in the fourth line.)
\end{proof}

\begin{lemma}\label{lem:Yball2}
 For every $\eps>0$, with probability $1-o_d(1)$, we have $\Ycal \subseteq B(\orig,2+\eps)$.
\end{lemma}

\begin{proof}
This follows from the fact that

$$ |\Ycal \setminus B(\orig,2+\eps)| = |\Zcal \cap \clo(B(\orig,2/(2+\eps))) | \isd 
\Po\left( \left(2/(2+\eps)\right)^d \right), $$ 

\noindent
where $\clo(.)$ denotes topological closure. We see that 

$$ \Pee( \Ycal \not\subseteq B(\orig,2+\eps) ) \leq \left(\frac{2}{2+\eps}\right)^d = o_d(1). $$

\end{proof}

For the remainder of this section, we set 

\begin{equation}\label{eq:rdef}
 r := \frac{4}{\sqrt{5}}, \quad \alpha := \arcsin(r^2/4).
\end{equation}

\noindent
(Note that $r<2$ and $\alpha < \pi/2$.)
We separate out the following observation as a lemma.

\begin{lemma}\label{lem:cosalpha2}
With $r,\alpha$ as given by~\eqref{eq:rdef}, we have $\cos(\alpha/2) = r/2$.
\end{lemma}

\begin{proof}
By the double angle formula

$$ 4/5 = r^2/4 = \sin(\alpha) = 2\sin(\alpha/2)\cos(\alpha/2) 
= 2 \cos(\alpha/2) \sqrt{1-\cos^2(\alpha/2)}. $$

\noindent
We see that we must have either $\cos(\alpha/2) = 1/\sqrt{5}$ or $\cos(\alpha/2) = 2/\sqrt{5}$.
We can exclude the first possibility, as $\cos(\alpha/2) = 1/\sqrt{5} < 1/\sqrt{2}$ would imply that 
$\alpha/2 > \pi/4$, which in turn would imply that $\alpha > \pi/2$, 
but clearly $\alpha = \arcsin(4/5) \in (0,\pi/2)$.
It follows that $\cos(\alpha/2) = 2/\sqrt{5} = r/2$, as stated by the lemma.
\end{proof}

\begin{lemma}\label{lem:notinconv}
For all $\eps > 0$ there is a $\delta=\delta(\eps)>0$ such that the following holds.
If $V \subseteq \eR^d$ and $v \in \eR^d$ are such that 
\begin{enumerate}
 \item $V \subseteq B(\orig,2+\delta)$, and;
 \item $\norm{v} > r+\eps$, and;
 \item $\angle v{\orig}w \geq \alpha/2$ for all $w \in V$ with $\norm{w} > r$, 
\end{enumerate}
then $v \not\in \conv(V)$.
\end{lemma}

\begin{proof}
Applying a suitable isometry, we can assume without loss of generality that $v = (x,0,\dots,0)$ with $x>r+\eps$. 
We consider the projection $\pi : \eR^d \to \eR$ onto the first coordinate.

For every $w \in V$ with $\norm{w} \leq r$ we of course have $\pi(w) \leq r$ as well.
For $w \in V$ with $\norm{w} > r$ we have

$$ \pi(w) = \cos\left( \angle v{\orig}w \right) \cdot \norm{w} \leq \cos(\alpha/2) \cdot (2+\delta)
= r \cdot (1+\delta/2), $$

\noindent
where we've used Lemma~\ref{lem:cosalpha2} and the specific choice of $r,\alpha$.
Having chosen $\delta := 2\eps/r$ appropriately, see that 

$$ \pi[ \conv(V) ] \subseteq (-\infty,r+\eps]. $$

\noindent 
So $x = \pi(v) \not\in \pi[ \conv(V) ]$ and in particular $v\not\in \conv(V)$ as well.
\end{proof}

\begin{lemma}\label{lem:prev}
 For all $\eps>0$, with probability $1-o_d(1)$, we have 
 
 $$ \sup_{u \in S^{d-1}} \left( \psi(u,\Vtyp^{\circ}) + \psi(-u,\Vtyp^{\circ}) \right) \leq 2 + r + \eps. $$
 
\end{lemma}

\begin{proof}
Let $\eps>0$ be arbitrary and let $\delta=\delta(\eps)$ be as provided by Lemma~\ref{lem:notinconv}.
We assume that $\Ycal$ is such that the conclusions of Lemmas~\ref{lem:angleswhp} and~\ref{lem:Yball2} hold  
(which happens with probability $1-o_d(1)$), with 
$\min(\delta,\eps)$ taking the role of $\eps$ in Lemma~\ref{lem:Yball2}
and $r,\alpha$ as specified by~\eqref{eq:rdef}.

Aiming for a contradiction, suppose that there is some $u \in S^{d-1}$ such that both 
$\psi(u,\Vtyp^{\circ})$ and $\psi(-u,\Vtyp^{\circ})$ are $> r+\eps$.
In other words, we can find $\lambda,\mu > r+\eps$ such that the points $v := \lambda u, w := -\mu u$
satisfy $v,w \in \Vtyp^{\circ} = \conv(\Ycal)$.
There must be a $y\in\Ycal$ such that 
$\norm{y} > r$ and $\angle y\orig v\leq \alpha/2$, because otherwise
by Lemma~\ref{lem:notinconv} (with $V=\Ycal$) would imply that $v \not\in \Vtyp^{\circ}$, contradicting the 
choice of $v$.
Analogously, there is a $z\in \Zcal$ with $\norm{z}>r$ and $\angle z\orig w\leq \alpha/2$.

We now point out that $\angle v\orig w = \pi$, which gives $\angle y\orig z \geq \pi-\alpha$, contradicting 
our assumption that the conclusion of Lemma~\ref{lem:angleswhp} holds.
It follows that, for every $u \in S^{d-1}$ at least one of $\psi(u,\Vtyp^{\circ}) \leq r+\eps$ or $\psi(-u,\Vtyp^{\circ}) \leq r+\eps$
holds.
Hence,
 
$$ \sup_{u \in S^{d-1}} \left( \psi(u,\Vtyp^{\circ}) + \psi(-u,\Vtyp^{\circ}) \right) \leq (2 + \eps) + (r + \eps) = 2 + r + 2\eps. $$

\noindent
Adjusting the value of $\eps$, the lemma follows.
\end{proof}

\begin{corollary}
 For all $\eps > 0$, with probability $1-o_d(1)$, we have 
 
 $$ \width(\Vtyp) > \frac{4}{2+r} - \eps. $$
 
\end{corollary}

\begin{proof}
Applying Lemma~\ref{lem:phipsi} (and recalling that $\Vtyp$ is almost surely a polytope with $\orig \in \inter(\Vtyp)$), 
almost surely, for every $u \in S^{d-1}$ we have 

$$ \begin{array}{rcl} 
w(u,\Vtyp)
& = & \displaystyle 
\frac{1}{\psi(u,\Vtyp^{\circ})} + \frac{1}{\psi(-u,\Vtyp^{\circ})} \\[2ex]
& = & \displaystyle  
\frac12 \cdot \left( \frac{2}{\psi(u,\Vtyp^{\circ})} + \frac{2}{\psi(-u,\Vtyp^{\circ})} \right) \\[2ex]
& \geq & \displaystyle 
\frac{2}{ \frac12 \left( \psi(u,\Vtyp^{\circ}) + \psi(-u,\Vtyp^{\circ}) \right) } \\[2ex]
& = & \displaystyle  
\frac{4}{ \psi(u,\Vtyp^{\circ}) + \psi(-u,\Vtyp^{\circ}) },
\end{array} $$

\noindent
using the convexity of $x \mapsto 2/x$ (or the arithmetic-harmonic inequality).
Hence

$$ \width(\Vtyp) = \inf_{u\in S^{d-1}} w(u,\Vtyp) \geq \frac{4}{\sup_{u\in S^{d-1}}  \left(\psi(u,\Vtyp^{\circ}) + \psi(-u,\Vtyp^{\circ})\right) }. $$

The result thus follows from Lemma~\ref{lem:prev}, adjusting the value of $\eps$.
\end{proof}

\subsection{Faces of large co-dimension.\label{sec:smallface}}

Having already obtained Lemma~\ref{lem:vertub}, in order to prove Proposition~\ref{prop:radiussmallface}
it suffices to show:

\begin{lemma} Let $\eps > 0$ be fixed, and let $k=k(d)$ satisfy $d-k\to\infty$.
 With probability $1-o_d(1)$, we have $\min\left\{ \norm{x} : x \in \bigcup \Fcal_k(\Vtyp)  \right\} > 1-\eps$.
\end{lemma}

The proof is essentially the same as that of Lemma~\ref{lem:vertlb}. For completeness we spell it out.

\begin{proof}
We use the observation that if $x \in F$ for some $F \in \Fcal_k(\Vtyp)$ then there must be 
$z_1,\dots,z_{d-k} \in \Zcal$ such that 
$z_1,\dots,z_{d-k} \in \partial B(x,\norm{x})$.
If there exists any such $x \in B(\orig,1-\eps)$ then there exists a cube $q \in \Qcal_\delta$ with 
$x \in q \subseteq B(\orig,1-\eps+\delta)$, where $\delta=\delta(\eps)>0$ is a constant, to be chosen sufficiently small.
If $c$ is any corner of $q$ then $B(c,1-\eps+2\delta)$ contains at least $d-k$ points of $\Zcal$.
We arrive at:

$$ \begin{array}{rcl} 
\Pee\left( \min\left\{ \norm{x} : x \in \bigcup \Fcal_k(\Vtyp)  \right\} > 1-\eps \right)
& \leq & 
N_{1,\delta} \cdot \Pee\left( \Po( (1-\eps+2\delta)^d ) \geq d-k \right) \\
& \leq & 
e^{O(d)} \cdot e^{-\Omega( d(d-k) )} = o_d(1), 
\end{array} $$

\noindent 
again using~\eqref{eq:Nub} and Lemma~\ref{lem:chernoff}.
\end{proof}

Corollary~\ref{cor:diamsmallface} follows immediately from the following observation and Proposition~\ref{prop:radiussmallface}.

\begin{lemma}\label{lem:diamconvinann}
If $F \subseteq B(\orig,1+\eps)\setminus B(\orig,1-\eps)$ is convex then $\diam(F) \leq 4\sqrt{\eps}$.
\end{lemma}

\begin{proof}
Let $p,q \in F$ be two arbitrary points. The midpoint $m := (p+q)/2$ also lies on $F$.
We write $\alpha := \angle p\orig m, \beta := \angle q\orig m$. Since $m$ is the middle of the 
line segment between $p$ and $q$, we have $\alpha+\beta=\pi$. One of $\alpha,\beta$ is $\geq \pi/2$.
Without loss of generality it is $\alpha$.
The cosine rule gives:

$$ \norm{p}^2 = \norm{m}^2 + \norm{p-m}^2 - 2 \norm{m}\norm{p-m}\cos\alpha
\geq \norm{m}^2 + \norm{p-m}^2. $$

\noindent
The assumptions on $\eps, F$ now give:

$$ \norm{p-m}^2 \leq \norm{p}^2 - \norm{m}^2 
\leq (1+\eps)^2 - (1-\eps)^2 = 4\eps. $$

Hence 

$$\norm{p-q} = 2\norm{p-m} \leq 4\sqrt{\eps}. $$

\noindent
Since $p,q \in F$ were arbitrary, this shows that $\diam(F) \leq 4\sqrt{\eps}$.
\end{proof}

\subsection{Faces of constant co-dimension\label{sec:regsimpl}}

In this section, we will prove Theorems~\ref{thm:diambigface} and~\ref{thm:fass} and Proposition~\ref{prop:Nkeps}.
Before we can start in earnest, we need some more definitions and 
preliminary observations.

For the remainder of the section, we fix $k \in \eN$.
For notational convenience, we'll write:

$$ v_k := \left(\frac{(k+1)^{k+1}}{k^{k}}\right)^{1/2}, \hspace{8ex} \ell_k := \left(\frac{2(k+1)}{k}\right)^{1/2}. $$

\noindent
As mentioned earlier, $v_k$ equals $k!$ times the volume and $\ell_k$ equals side-length of 
a regular simplex inscribed in $S^{k-1}$.
We leave the elementary considerations verifying this to the reader.  
(Computing $v_k$ for instance occurs as an exercise with a difficulty rating of 2
in Section 13 of~\cite{MatousekDiscrGeom} and 
the value of $v_k$ is simply stated in Section 49 of~\cite{FejestothRegFig}, without further elaboration.
Determining $\ell_k$ is even easier, at least in the opinion of the 
authors.)

Also for notational convenience, we write for $u_1,\dots,u_k \in\eR^k$:

\begin{equation}\label{eq:Ddef} 
D(u_1,\dots,u_k) := \left|\det\left(u_1|\ldots|u_k\right)\right|, 
\end{equation}

\noindent
where $(u_1|\ldots|u_k)$ denotes the $k\times k$ matrix whose columns are $u_1,\dots,u_k$.
(Note we consider $k$ vectors in $k$-dimensional Euclidean space.)
We denote by $P(u_1,\dots,u_k)$ the {\em parallelopiped spanned by $u_1,\dots,u_k \in \eR^k$}. That is:

$$ P(u_1,\dots,u_k) := \{ \lambda_1 u_1 + \dots + \lambda_k u_k : 0 \leq \lambda_1 \leq 1, \dots, 0\leq \lambda_k \leq 1 \}. $$

\noindent
A standard elementary fact states that:

\begin{equation}\label{eq:elem} \begin{array}{rcl} 
D(u_1,\dots,u_k)
& = & \vol( P(u_1,\dots,u_k) ) \\
& = & k! \cdot \vol\left( \conv\left(\left\{\orig,u_1,\dots,u_k\right\}\right)\right). 
  \end{array} \end{equation}

\noindent
(This last identity is for instance stated as equation (7.6) in~\cite{SchneiderWeil}.)

For notational convenience, for $x_1,\dots,x_k \in\eR^d$, we write 

$$ \rho(x_1,\dots,x_k) := \inf\left\{ r > 0 : 
\text{$\exists x$ such that $\orig,x_1,\dots,x_k \in \partial B(x,r)$}\right\}. $$

\noindent
(Note that now the ambient dimension $d$ is not necessarily the same as the number of points $k$.)
We collect some observations on $\rho$ needed in the sequel.

\begin{lemma}\label{lem:rho}
\text{ }
\begin{enumerate}
 \item\label{itm:ballexist} If $x_1,\dots,x_k \in \eR^d$ are linearly independent, then $0<\rho(x_1,\dots,x_k)<\infty$ and
 there is a unique open ball $B$ of radius $\rho(x_1,\dots,x_k)$ such that $\orig,x_1,\dots,x_k \in \partial B$.
 \item\label{itm:rhocont} The map $(x_1,\dots,x_k) \mapsto \rho(x_1,\dots,x_k)$ is continuous on the 
 set $I_{k,d} \subseteq \eR^{kd}$ given by
 
 $$ I_{k,d} := \left\{  (x_1,\dots,x_k) \in \eR^d \times \dots\times\eR^d : x_1,\dots,x_k
 \text{ are linearly independent}\right\}. $$ 
 
\end{enumerate}
\end{lemma}

\begin{proof}
A ball that has $\orig, x_1,\dots,x_k$ on its boundary has to be of the form $B(x,\norm{x})$ with $x$ satisfying

$$ \norm{x} = \norm{x-x_1} = \dots = \norm{x-x_k}. $$

\noindent
Squaring and rewriting the squared norm in terms of the inner product gives

$$ \langle x,x\rangle =  \langle x-x_i,x-x_i\rangle = 
\langle x,x\rangle - 2\langle x_i,x\rangle + \langle x_i,x_i\rangle \quad (i=1,\dots,k). $$

\noindent 
Reorganising gives:

$$ \langle x_i,x \rangle = \norm{x_i}^2/2 \quad (i=1,\dots,k). $$

\noindent
Writing $A$ for the $k\times d$ matrix whose $i$-th row is $x_i^t$ and $b \in \eR^k$ for the 
vector whose $i$-th entry equals $\norm{x_i}^2/2$, we see that $x$ must satisfy

\begin{equation}\label{eq:Axisb} 
A x = b,  
\end{equation}

\noindent
and any $x$ satisfying this equation defines a ball of the sought form.
If $x_1,\dots,x_k$ are linearly independent then the $k\times k$ matrix $A A^t$ is non-singular. 
In particular $(A A^t)^{-1}$ is well defined. So if $x_1,\dots,x_k$ are linearly independent then

\begin{equation}\label{eq:xsol} 
x := A^t(AA^t)^{-1} b
\end{equation}

\noindent 
solves~\eqref{eq:Axisb}.
Note that $x$ is a linear combination of the columns $A^t$. 
In other words, it lies in the linear hull $\Lcal(\{x_1,\dots,x_k\})$ of $x_1,\dots,x_k$. 

Now let $y$ be an arbitrary solution of~\eqref{eq:Axisb}.
We have $A (y-x) = \orig$, so that $\langle y-x, x_i\rangle = 0$ for $i=1,\dots, k$.
But then we also have $\langle y-x,x\rangle = 0$.
Pythagoras's theorem tells us that

$$ \norm{y}^2 = \norm{x}^2 + \norm{y-x}^2. $$

\noindent
So $x$ is the (unique) solution with the smallest possible norm. This proves~\ref{itm:ballexist}.

To see~\ref{itm:rhocont} we note that, using the Leibniz formula for the determinant 
and Cramer's rule for the matrix inverse, each individual coordinate of $x = x(x_1,\dots,x_k)$
can be written as $P(x_1,\dots,x_k)/Q(x_1,\dots,x_k)$ where $P,Q$ are polynomials in the 
coordinates of $x_1,\dots,x_k$ and specifically $Q = \det(AA^t)$.
The polynomial $Q$ is non-zero on $I_{kd}$, so that $x$ is a continuous function of $x_1,\dots,x_k$ on $I_{kd}$.
So $\rho(x_1,\dots,x_k)
= \norm{x}$ is also continuous on $I_{k,d}$.
\end{proof}

The unique ball that has $\orig,x_1,\dots,x_k$ on its boundary, guaranteed to exist 
when $x_1,\dots,x_k$ are linearly independent, 
will be denoted by $B(x_1,\dots,x_k)$.
We shall only consider $B(x_1,\dots,x_k)$ for $x_1,\dots,x_k$ distinct elements of the Poisson point process $\Zcal$.
As long as $k \leq d$, almost surely any set of $k$ distinct points of $\Zcal$ is linearly independent.

Another observation we'll use is that for $\mu>0$ and $u_1,\dots,u_k \in \eR^k$ we have 

\begin{equation}\label{eq:rescale} \rho(\mu u_1, \dots, \mu u_k ) = \mu \rho(u_1,\dots,u_k), \quad 
D(\mu u_1,\dots,\mu u_k)
= \mu^k D(u_1,\dots,u_k). 
\end{equation}

We set:

\begin{equation}\label{eq:Xdef} 
\Xcal(r) := \left\{ (z_1,\dots,z_k) \in \Zcal^k : \begin{array}{l} z_1,\dots,z_k \text{ are distinct, and } \\
\rho(z_1,\dots,z_k) \leq r \end{array} \right\}, \quad X(r) := \left|\Xcal(r)\right|,
\end{equation}

\noindent
Applying the Mecke formula we obtain

\begin{equation}\label{eq:Xmecke}
\Ee X(r) = \lambda^k \int_{\eR^d}\dots\int_{\eR^d} 1_{\{\rho(x_1,\dots,x_k) \leq r\}}  \dd x_1\dots\dd x_k. 
\end{equation}

\noindent
The following approximations will be reused a few times.

\begin{lemma}\label{lem:Iint} 
$\text{ }$
\begin{enumerate}
 \item\label{itm:Iint} For $r>0$ fixed and $d\to\infty$, we have $\displaystyle\sqrt[d]{\Ee X(r)} = r^k v_k + o_d(1)$.
 \item\label{itm:Iint2} For $r>1/\sqrt{2}$ fixed and $d\to\infty$, we have 
 $\displaystyle\Ee X(r)^2 = (1+o_d(1)) \cdot \left(\Ee X(r)\right)^2$.
\end{enumerate}
\end{lemma}

\noindent
The proof of Lemma~\ref{lem:Iint} makes use of the following observations.

\begin{lemma}\label{lem:det}
$\text{ }$
\begin{enumerate}
\item\label{itm:det} For every $k \in \eN$, we have 
$$
\max_{u_1,\dots,u_k \in \eR^k, \atop \rho(u_1,\dots,u_k)\leq 1} 
D(u_1,\dots,u_k)
= v_k. $$

\noindent
Moreover, the maximum is attained by $u_1,\dots,u_k \in \eR^k$ 
if and only if $\orig,u_1,\dots,u_k$ form the vertices of a regular simplex of side-length $\ell_k$.

\item\label{itm:detgen} For every $k,\ell \in \eN$, we have 

$$
\max_{{\tiny \begin{array}{l} u_1,\dots,u_{k+\ell} \in \eR^{k+\ell},\\\rho(u_1,\dots,u_{k})\leq 1,\\
\rho(u_{k+1},\dots,u_{k+\ell})\leq 1\end{array}}} D(u_1,\dots,u_{k+\ell})
\leq  v_k\cdot v_\ell. $$

\end{enumerate}
\end{lemma}

Before we get to the proof of Lemma~\ref{lem:Iint}, we first deal with Lemma~\ref{lem:det}.
We make use of a ``folklore'' fact about simplices inscribed in the unit ball.

\begin{proposition}\label{prop:slepian}
Among all simplices inscribed into the unit sphere, only the regular simplices maximize the volume. More precisely, for
$w_0,\dots,w_k \in S^{k-1}$ we have 

$$ \vol\left(\conv\left(\left\{w_0,\dots,w_k\right\}\right)\right) \leq v_k/k!, $$

\noindent
with equality if and only if $\norm{w_i-w_j}=\ell_k$ for all $0\leq i<j\leq k$.
\end{proposition}

This result can be found in the literature. 
It is for instance derived in~\cite{Slepian69} and (independently)~\cite{Tanner74}. We believe it is significantly
older, but have not managed to find an earlier, explicit reference. 
In both~\cite{Slepian69} and~\cite{Tanner74} the proposition is a consequence of a more general result. 
There is a much shorter and simpler proof of Proposition~\ref{prop:slepian} however. 
See for instance the start of the proof of Lemma 13.2.2 in~\cite{MatousekDiscrGeom}.

\begin{proofof}{Lemma~\ref{lem:det}}
The relations~\eqref{eq:rescale} show that when computing the sought maximum we can restrict attention 
to $k$-tuples $u_1,\dots,u_k$ such that 
$\rho(u_1,\dots,u_k)=1$.

Let $u_1,\dots,u_k \in \eR^k$ be such that $\rho(u_1,\dots,u_k)=1$.
Let $B$ be the (unique) ball of radius one such that $\orig, u_1,\dots,u_k \in \partial B$ and let $c$ denote its center.
If we set $w_i = u_i-c$ for $i=0,\dots,k$ where $u_0 := \orig$, then $w_0,\dots,w_k \in S^{k-1}$.
In fact, this procedure provides a one-to-one correspondence between 
$(k+1)$-tuples $w_0,\dots,w_k \in S^{k-1}$ and $k$-tuples $u_1,\dots,u_k \in \eR^k$ satisfying $\rho(u_1,\dots,u_k)=1$.
We also have

$$ \begin{array}{rcl} 
\displaystyle D(u_1,\dots,u_k)
& = & \displaystyle k! \cdot \vol\left(\conv\left(\left\{\orig,u_1,\dots,u_k\right\}\right)\right) \\
& = & \displaystyle k! \cdot \vol\left(\conv\left(\left\{w_0,\dots,w_k\right\}\right)\right). 
\end{array} $$

\noindent
Part~\ref{itm:det} now follows immediately from Proposition~\ref{prop:slepian}.

For the proof of part~\ref{itm:detgen} we use the first interpretation of the absolute value of the determinant 
given in~\eqref{eq:elem}.
The volume of the parallelopiped $P(u_1,\dots,u_{k+\ell})$ can be written as 

$$ \begin{array}{rcl} D(u_1,\dots,u_{k+\ell}) 
& = & \displaystyle 
\prod_{i=1}^{k+\ell} \dist( u_i, \Lcal(\{u_1,\dots,u_{i-1}\}) ) \\
& \leq & \displaystyle 
\left( \prod_{i=1}^{k} \dist( u_i, \Lcal(\{u_1,\dots,u_{i-1}\}) ) \right) \cdot 
\left( \prod_{j=1}^{\ell} \dist( u_{k+j}, \Lcal(\{u_{k+1},\dots,u_{k+j-1}\}) \right)
\end{array} $$

\noindent
where 
$\dist(x,A) := \inf_{y \in A} \norm{x-y}$.
Writing $v_{k,\ell}$ for the maximum in the LHS in Part~\ref{itm:detgen} of the lemma statement, we find:

$$ \begin{array}{rcl} 
v_{k,\ell}
& \leq & \displaystyle 
\left( \max_{u_1,\dots,u_k \in \eR^{k+\ell},\atop\rho(u_1,\dots,u_{k})\leq 1}
\prod_{i=1}^{k} \dist( u_i, \Lcal(\{u_1,\dots,u_{i-1}\}) ) \right) \cdot
\left( \max_{w_1,\dots,w_\ell \in \eR^{k+\ell},\atop\rho(w_1,\dots,w_{\ell})\leq 1}
\prod_{i=1}^{\ell} \dist( w_i, \Lcal(\{w_1,\dots,w_{i-1}\}) ) \right) \\
& = & \displaystyle 
\left( \max_{u_1,\dots,u_k \in \eR^{k},\atop\rho(u_1,\dots,u_{k})\leq 1}
\prod_{i=1}^{k} \dist( u_i, \Lcal(\{u_1,\dots,u_{i-1}\}) ) \right) \cdot
\left( \max_{w_1,\dots,w_\ell \in \eR^{\ell},\atop\rho(w_1,\dots,w_{\ell})\leq 1}
\prod_{i=1}^{\ell} \dist( w_i, \Lcal(\{w_1,\dots,w_{i-1}\}) ) \right) \\
& = & v_k \cdot v_\ell,
   \end{array} $$

\noindent
where in the penultimate line we use that if we apply an orthogonal transformation $T$
mapping $u_1,\dots,u_k \in \eR^{k+\ell}$ to $Tu_1,\dots,Tu_k \in \eR^k \times \{0\}^\ell$ then 
$\rho(Tu_1,\dots,Tu_k) = \rho(u_1,\dots,u_k)$ and 
$\dist( Tu_i, \Lcal(\{Tu_1,\dots,Tu_{i-1}\}) ) = \dist( u_i, \Lcal(\{u_1,\dots,u_{i-1}\}) )$.
\end{proofof}

\begin{proofof}{Lemma~\ref{lem:Iint}}
Applying the linear Blaschke-Petkantschin 
formula (see e.g.~\cite{SchneiderWeil}, Theorem 7.2.1) to~\eqref{eq:Xmecke}, we have:

\begin{equation}\label{eq:aap} 
\begin{array}{rcl} 
\Ee X(r)
& = & \displaystyle 
\lambda^k \cdot \frac{(d+1)\cdot\dots\cdot(d-k+2)}{(k+1)\cdot\dots\cdot 2} \cdot 
\frac{\kappa_{d+1}\cdot\dots\cdot\kappa_{d-k+2}}{\kappa_{k+1}\cdot\dots\cdot\kappa_2} \cdot \\
& & \displaystyle \hspace{5ex}
\int_{\eR^k}\dots\int_{\eR^k} 
1_{\{\rho(u_1,\dots,u_k) \leq r\}}\cdot D(u_1,\dots,u_k)^{d-k} \dd u_1\dots\dd u_k,
\end{array} 
\end{equation}

\noindent
(To obtain this identity from Theorem 7.2.1 of~\cite{SchneiderWeil}, we use the identities 
$\omega_i = (i+1)\kappa_{i+1}$ and that 
$\rho(.)$ is invariant under orthogonal transformations. That is $\rho(Tx_1,\dots,Tx_k) = \rho(x_1,\dots,x_k)$
for any orthogonal transformation $T : \eR^d\to\eR^d$ and any $x_1,\dots,x_k$. 
Note that in~\cite{SchneiderWeil}, 
the inner integral in the RHS of (7.7) does not depend on the choice of $L \in G(d,q)$ and $\nu_q(.)$ denotes 
the uniform measure on $G(d,q)$, the space of all linear subspaces of dimension $q$.) 

The choice $\lambda = 1/\kappa_d$ implies that 
the constant in the RHS of \eqref{eq:aap} equals

\begin{equation}\label{eq:kappas} 
\frac{1}{(k+1)!} \cdot \frac{1}{\prod_{i=2}^{k+1} \kappa_i} \cdot \frac{\kappa_{d+1}}{\kappa_d} \cdot\dots\cdot 
\frac{\kappa_{d-k+2}}{\kappa_d} = d^{O(1)}, 
\end{equation}

\noindent 
using~\eqref{eq:volball} and  that $\Gamma(t+\alpha)/\Gamma(t) = (1+o_t(1)) \cdot t^\alpha$ 
if $\alpha \in \eR$ is fixed and $t\to\infty$, as can for instance be seen from Stirling's approximation to 
the Gamma function.

If $\norm{u_i} > 2r$ for some $1\leq i \leq k$ then $\rho(u_1,\dots,u_k)>r$.
Hence 

\begin{equation}\label{eq:EEXX} 
\Ee X(r) = d^{O(1)} \cdot \int_{B_{\eR^k}(\orig,2r)}\dots\int_{B_{\eR^k}(\orig,2r)} 
1_{\{\rho(u_1,\dots,u_k) \leq r\}}\cdot D(u_1,\dots,u_k)^{d-k} \dd u_1\dots\dd u_k. 
\end{equation}

\noindent
By Lemma~\ref{lem:det} and~\eqref{eq:EEXX},~\eqref{eq:rescale}, we have

$$ \Ee X(r) \leq d^{O(1)} \cdot \left( \kappa_k (2r)^k \right)^k \cdot \left( r^k v_k \right)^{d-k}. $$

\noindent 
Taking $d$-th roots, we find 

$$ \sqrt[d]{\Ee X(r)} \leq (1+o_d(1)) \cdot r^k v_k = r^k v_k + o_d(1). $$

It remains to derive a lower bound of the same form. To this end, we fix an arbitrarily small $\eta>0$ and 
let $A \subseteq \eR^{k^2}$ be defined by

$$ A := \left\{ (u_1,\dots,u_k) \in \eR^k \times\dots\times\eR^k :
\begin{array}{l} \rho(u_1,\dots,u_k) < r,\text{ and}\\
D(u_1,\dots,u_k)>(1-\eta) r^k v_k 
\end{array}\right\}. $$

\noindent
We have

\begin{equation}\label{eq:EeXrlb} 
\Ee X(r) \geq d^{O(1)} \cdot \vol_{k^2}(A) \cdot \left((1-\eta)r^k v_k\right)^{d-k}. 
 \end{equation}

The set $A$ is non-empty. (If $\orig,u_1,\dots,u_k$ form a regular simplex of
side-length $(1-\eta/2)^{1/k} r \ell_k$, then $(u_1,\dots,u_k) \in A$ by~\eqref{eq:rescale}
and Lemma~\ref{lem:det}.) 
The set $A$ is also open. (To see this, note that $D$ is continuous, so that $\{D>(1-\eta)r^kv_k\}$
is open. Now note that $\rho$ is continuous on $\{D>(1-\eta)r^kv_k\}$.)
Being non-empty and open implies that $A$ has positive $k^2$-dimensional volume.
So, taking the $d$-th root of~\eqref{eq:EeXrlb} gives 

$$ \sqrt[d]{\Ee X(r)} \geq (1+o_d(1)) \cdot (1-\eta) r^k v_k. $$

\noindent
Since $\eta>0$ was arbitrary this in fact gives $\sqrt[d]{\Ee X(r)} \geq r^k v_k +o_d(1)$, finishing the
proof of~\ref{itm:Iint}.

We now turn attention to the second moment of $X(r)$.
We remark that

\begin{equation}\label{eq:Wdef} 
X(r)^2 = \left|\left\{ (z_1,\dots,z_k, w_1,\dots,w_k) \in \Zcal^{2k} : \begin{array}{l} 
z_1,\dots,z_k \text{ are distinct, and} \\
w_1,\dots,w_k \text{ are distinct, and} \\
\rho(z_1,\dots,z_k) \leq r, \text{ and} \\
\rho(w_1,\dots,w_k) \leq r.\end{array}\right\}\right|.
\end{equation}

\noindent
We define for $\ell=0,\dots,k$:

$$
X_{\ell} := \left|\left\{ (z_1,\dots,z_{k+\ell}) \in \Zcal^{k+\ell} : 
\begin{array}{l} 
z_1,\dots,z_{k+\ell} \text{ are distinct, and} \\
\rho(z_1,\dots,z_k) \leq r, \text{ and} \\ 
\rho(z_1,\dots,z_{k-\ell}, z_{k+1},\dots,z_{k+\ell}) \leq r.
\end{array}\right\}\right|.
$$

\noindent
Accounting for the ways in which $(z_1,\dots,z_k)$ and $(w_1,\dots,w_k)$ in~\eqref{eq:Wdef} might overlap 
and using symmetry, we find 

\begin{equation}\label{eq:EEW} 
\Ee X(r)^2 = \sum_{\ell=0}^k (k-\ell)! \cdot {k\choose\ell}^2 \cdot \Ee X_\ell. 
\end{equation}

We will bound $\Ee X_\ell$ separately for each $\ell$.
The easiest value is of course $\ell=0$, as $X_0=X(r)$. 
In particular, Part~\ref{itm:Iint} tells us that $\Ee X_0 = \Ee X(r) = \left( r^k v_k + o_d(1)\right)^d$.
The next simplest value is $\ell=k$. 
Now the Mecke formula gives

\begin{equation}\label{eq:EEWk} \begin{array}{rcl} 
\Ee X_k 
& = & \displaystyle 
\lambda^{2k} 
\int_{\eR^d}\dots\int_{\eR^d} 
1_{\left\{\rho(x_1,\dots,x_k),\rho(x_{k+1},\dots,x_{2k}) \leq r\right\}}\dd x_1\dots\dd x_{2k} \\[2ex]
 & = & \displaystyle 
 \left( \lambda^k \int_{\eR^d}\dots\int_{\eR^d} 1_{\{\rho(x_1,\dots,x_k) \leq r\}}\dd x_1\dots\dd x_k\right)^2 \\[2ex]
 & = &       
 \Ee X(r)^2. 
 \end{array} \end{equation}

\noindent
For the remaining values $0<\ell<k$ a little more work is needed.
Arguing as in the proof of Lemma~\ref{lem:Iint}, we have:

\begin{equation}\label{eq:EeXl} \begin{array}{rcl} 
\Ee X_\ell 
& = & \displaystyle 
\lambda^{k+\ell} 
\int_{\eR^d}\dots\int_{\eR^d} 
1_{\left\{\rho(x_1,\dots,x_k)\leq r,\rho(x_{1},\dots,x_{k-\ell},x_{k+1},\dots,x_{k+\ell}) \leq r\right\}}
\dd x_1\dots\dd x_{k+\ell} \\[2ex]
& \leq & \displaystyle 
\lambda^{k+\ell} 
\int_{\eR^d}\dots\int_{\eR^d} 
1_{\left\{\rho(x_1,\dots,x_k) \leq r,\rho(x_{k+1},\dots,x_{k+\ell}) \leq r\right\}}
\dd x_1\dots\dd x_{k+\ell} \\[2ex]
 & = & \displaystyle 
 d^{O(1)} 
 \cdot \int_{B_{\eR^{k+\ell}}(\orig,2)}\dots\int_{B_{\eR^{k+\ell}}(\orig,2)} 1_{\{\rho(u_1,\dots,u_k)\leq r,
 \rho(u_{k+1},\dots,u_{k+\ell})\leq r\}} \\[2ex]
 & & \displaystyle \hspace{35ex} \cdot
 D(u_1,\dots,u_{k+\ell})^{d-(k+\ell)}
 \dd u_1\dots\dd u_{k+\ell} \\[2ex]
& \leq & \displaystyle 
d^{O(1)} \cdot \left( \kappa_{k+\ell} 2^{k+\ell} \right)^{k+\ell}
\cdot \left( r^{k+\ell} v_k v_\ell \right)^{d-(k+\ell)} \\[2ex]
& = & \left( r^{k+\ell} v_k v_\ell +o_d(1) \right)^d, 
\end{array} \end{equation}

\noindent 
where we used Part~\ref{itm:detgen} of Lemma~\ref{lem:det} in the penultimate line.

We have $v_1=2$ and for $\ell \geq 2$:

$$ \begin{array}{rcl}
\displaystyle 
\left(\frac{v_\ell}{v_{\ell-1}}\right)^2
& = & \displaystyle 
\frac{(\ell+1)^{\ell+1}(\ell-1)^{\ell-1}}{\ell^\ell}
= \frac{(\ell+1)^2}{\ell} \cdot \left(\frac{(\ell+1)(\ell-1)}{\ell^2}\right)^{(\ell-1)/2} \\[2ex]
& = & \displaystyle 
\frac{\ell^2+2\ell+1}{\ell} \cdot \left(1-\frac{1}{\ell^2}\right)^{\ell-1}
\geq (\ell+2) \cdot \left(1-\frac{1}{\ell^2}\right)^\ell \\[2ex]
& \geq & \displaystyle 
(\ell+2) \cdot \left(1-\frac{1}{\ell}\right) \geq 
4 \cdot (1-1/2) = 2. 
\end{array} $$

\noindent
So if $r>1/\sqrt{2}$ then $1 < rv_1 < r^2v_2 < \dots < r^kv_k$. 
Thus, by Part~\ref{itm:Iint} and~\eqref{eq:EeXl}, 
provided $r>1/\sqrt{2}$:

$$ \Ee X_\ell = o\left( \Ee X(r)^2 \right) \quad (\ell = 0,\dots,k-1). $$

\noindent
Part~\ref{itm:Iint2} now follows by combining with~\eqref{eq:EEW} and~\eqref{eq:EEWk}.
\end{proofof}

\begin{proofof}{Theorem~\ref{thm:fass}}
We let $\eps>0$ be arbitrary and we let $\delta=\delta(\eps,k)>0$ be a small constant to be chosen more precisely 
in the remainder of the proof.
We note that if $z_1,\dots,z_k \in \Zcal$ define a $(d-k)$-face $F$ then
every $x \in F$ must satisfy $\norm{x} \geq \rho(z_1,\dots,z_k)$ as
$B(x,\norm{x})$ contains $\orig,z_1,\dots,z_k$ on its boundary.
Since every face contains a vertex, Lemma~\ref{lem:vertub} implies that, with probability $1-o_d(1)$, 
the only $k$-tuples $z_1,\dots,z_k \in \Zcal$ that define a $(d-k)$-face
must satisfy $\rho(z_1,\dots,z_k) < 1+\delta$. 
We have:

\begin{equation}\label{eq:fleqX}
\Pee\left( f_{d-k}(\Vtyp) \leq X(1+\delta) \right) = 1-o_d(1).
\end{equation}

\noindent
Provided $\delta=\delta(\eps,k)$ was chosen sufficiently small, Markov's inequality tells us that 

$$ \Pee\left( X(1+\delta) > (1+\eps)^d v_k^d \right) \leq 
\frac{\Ee X(1+\delta)}{(1+\eps)^d v_k^d} 
= \left( \frac{(1+\delta)^k v_k + o_d(1)}{(1+\eps)v_k}\right)^d
= o_d(1). $$

\noindent
Combining with~\eqref{eq:fleqX}, this gives 

\begin{equation}\label{eq:fassub} \begin{array}{rcl} 
\displaystyle 
\Pee\left( \sqrt[d]{f_{d-k}(\Vtyp)} > (1+\eps) v_k \right)
& \leq & \displaystyle 
\Pee\left( f_{d-k}(\Vtyp) > X(1+\delta) \right) \\
& & \displaystyle + \Pee\left( X(1+\delta) > (1+\eps)^d v_k^d \right) \\
& = & o_d(1). 
\end{array} \end{equation}

It remains to derive an upper bound on $\Pee\left( \sqrt[d]{f_{d-k}(\Vtyp)} < (1-\eps) v_k \right)$ that tends to zero with $d$.
For this purpose, we define 

\begin{equation}\label{eq:Ydef} 
Y := \left|\left\{ (z_1,\dots,z_k) \in \Zcal^k : \begin{array}{l} z_1,\dots,z_k \text{ are distinct, and} \\
\rho(z_1,\dots,z_k) \leq 1-\delta, \text{ and} \\
B(z_1,\dots,z_k) \cap \Zcal = \emptyset \end{array} \right\}\right|.
\end{equation}

\noindent
Clearly each $k$-tuple counted by $Y$ defines a $(d-k)$-face of $\Vtyp$. 
As each $k$-tuple has $k!$ re-orderings, we have

$$ f_{d-k}(\Vtyp) \geq \frac{1}{k!} \cdot Y. $$

\noindent
(We also use that, almost surely, 
no two distinct $k$-sets of points in $\Zcal$ define the same face.)
Using the Mecke formula once again we have 

$$ \Ee Y =
\lambda^k \int_{\eR^d}\dots\int_{\eR^d} 1_{\{\rho(x_1,\dots,x_k) \leq 1-\delta\}} 
\cdot \Pee( B(x_1,\dots,x_k) \cap \Zcal = \emptyset) \dd x_1\dots\dd x_k. $$

\noindent
Comparing to~\eqref{eq:Xmecke} we see that 
$e^{-(1-\delta)^d} \cdot \Ee X(1-\delta) \leq \Ee Y \leq \Ee X(1-\delta)$. 
In other words 

\begin{equation}\label{eq:noot}
\Ee Y = (1+o_d(1)) \cdot \Ee X(1-\delta).
\end{equation}

\noindent
Since $Y \leq X(1-\delta)$, provided $\delta=\delta(\eps,k)$ was chosen sufficiently small, 
Part~\ref{itm:Iint2} of Lemma~\ref{lem:Iint} now gives

$$ \Ee Y^2 \leq \Ee X(1-\delta)^2 = (1+o_d(1)) \cdot \left( \Ee Y \right)^2. $$

\noindent
By~\eqref{eq:noot} and Part~\ref{itm:Iint} of 
Lemma~\ref{lem:Iint}, provided $\delta=\delta(\eps,k)$ is chosen sufficiently small, we have 

$$ k! \cdot (1-\eps)^d v_k^d < \frac{1}{2} \Ee Y, $$

\noindent 
for $d$ sufficiently large.
We can now apply Chebyshev's inequality to obtain:

$$ \begin{array}{rcl} \Pee\left( \sqrt[d]{f_{d-k}(\Vtyp)} < (1-\eps) v_k \right) 
& \leq & \Pee( Y <  k! \cdot (1-\eps)^d v_k^d ) \\
& \leq & \Pee( |Y-\Ee Y| > \frac12 \Ee Y ) \\
& \leq & 4 \Var Y / \left(\Ee Y\right)^2 \\
& = & o_d(1). \end{array} $$

\noindent
Together with~\eqref{eq:fassub} this establishes Theorem~\ref{thm:fass}. 
\end{proofof}

\begin{proofof}{Theorem~\ref{thm:diambigface}}
Recall that if $F \in \Fcal_{d-k}(\Vtyp)$ is defined by $z_1,\dots,z_k \in \Zcal$ then 
$\norm{x} \geq \rho(z_1,\dots,z_k)$ for all $x \in F$.
So by Lemma~\ref{lem:vertub}, with probability $1-o_d(1)$, 
any face defined by some $z_1,\dots,z_k \in \Zcal$ with $\rho(z_1,\dots,z_k) \geq 1-\eps$ is contained in the annulus 
$B(\orig,1+\eps)\setminus B(\orig,1-\eps)$.
By Theorem~\ref{thm:main}, with probability $1-o_d(1)$, $\Vtyp$ has diameter at most $2+\eps$.
Applying Lemma~\ref{lem:diamconvinann} we obtain that, with probability $1-o_d(1)$:

$$ \sum_{F \in \Fcal_{d-k}(\Vtyp)} \diam(F) \leq 4\sqrt{\eps}\cdot f_{d-k}(\Vtyp) + (2+\eps)\cdot X(1-\eps). $$

\noindent
We note that 

$$ \Pee\left( X(1-\eps) > \eps f_{d-k}(\Vtyp) \right) \leq \Pee\left( X(1-\eps) > \eps(1-\eps/2)^dv_k^d \right)
+ \Pee\left( f_{d-k}(\Vtyp) < (1-\eps/2)^dv_k^d \right). $$

\noindent 
The first term in the RHS is $o_d(1)$ by Lemma~\ref{lem:Iint} and Markov's inequality.
The second term is $o_d(1)$ as well, by Theorem~\ref{thm:fass}.
We arrive at:

$$ \Pee\left( \frac{1}{f_{d-k}(\Vtyp)} \sum_{F \in \Fcal_{d-k}(\Vtyp)} \diam(F) > 
4\sqrt{\eps} + \eps(2+\eps) \right) = o_d(1). $$

\noindent
The theorem follows. 
\end{proofof}

It remains to prove Proposition~\ref{prop:Nkeps}.
We need one more preparatory lemma.

\begin{lemma}\label{lem:zevenendertig}
For every $\eps>0$ and $k\in\eN$ there exist $\delta>0$ and $c < v_k$ 
such that the following holds. 
For all $u_1,\dots,u_k \in \eR^k$ that are not $\eps$-near regular, we have 
$\rho(u_1,\dots,u_k) > 1+\delta$ or $D(u_1,\dots,u_k) \leq c$, or both.
\end{lemma}

\begin{proof}
Let us set

$$ \displaystyle g(u_1,\dots,u_k) := \sum_{i=1}^k \left( \norm{u_i}-\ell_k\right)^2 + \sum_{1\leq i<j \leq k}
 \left( \norm{u_i-u_j}-\ell_k\right)^2, $$
 
$$ W := \left\{ (u_1,\dots,u_k) \in \eR^k\times\dots\times\eR^k : 
\begin{array}{l} D(u_1,\dots,u_k) \geq v_k/2, \text{ and,} \\
\rho(u_1,\dots,u_k) \leq 1, \text{ and,} \\
g(u_1,\dots,u_k) \geq \eps^2/4
\end{array} \right\}, $$

\noindent 
and 

$$ c' := \max\left( \sup_{(u_1,\dots,u_k) \in W} D(u_1,\dots,u_k), \frac{v_k}{2} \right). $$

\noindent
We claim that $c' < v_k$. If $W = \emptyset$ then this is clearly true, so suppose $W$ is non-empty.
Since $D, g$ are continuous, the set $W_0 := \{ D \geq v_k/2, g \geq \eps^2/4 \}$ is closed.
By Lemma~\ref{lem:rho}, $\rho$ is continuous on $W_0$. So $W$ is closed too.
It is also bounded, since $\rho\leq 1$ implies $\norm{u_1},\dots,\norm{u_k} \leq 2$.
We see that $W$ is compact, and hence the supremum $c'' := \sup_{(u_1,\dots,u_k) \in W} D(u_1,\dots,u_k)$ is 
attained by some $(w_1,\dots,w_k) \in W$. Lemma~\ref{lem:det} tells us that 
$c'' < v_k$ as $g(w_1,\dots,w_k) \neq 0$ implies that $\orig,w_1,\dots,w_k$ is not a regular simplex
of side length $\ell_k$.
So $c' = \max( v_k/2, c'' ) < v_k$ as claimed.

We now choose a small $\delta=\delta(\eps,k) > 0$, small enough so that $(1+\delta)^k c' < v_k$ and 
$(\ell_k+\eps)/(1+\delta) > \ell_k + \eps/2$.
Let $u_1,\dots,u_k$ be such that $\rho(u_1,\dots,u_k) \leq 1+\delta$ and 
$u_1,\dots,u_k$ is not $\eps$-near regular, but otherwise arbitrary.
Setting $w_i := u_i/(1+\delta)$ we see that $\rho(w_1,\dots,w_k) \leq 1$ by~\eqref{eq:rescale}.
Moreover, if $i,j$ are such that $\norm{u_i-u_j} \geq \ell_k+\eps$  then $\norm{w_i-w_j} \geq \ell_k+\eps/2$.
Similarly, if $\norm{u_i} \geq \ell_k+\eps$  then $\norm{w_i} \geq \ell_k+\eps/2$.
We also have $\norm{w_i} \leq \norm{u_i}, \norm{w_i-w_j} \leq \norm{u_i-u_j}$. 
It follows that $g(w_1,\dots,w_k) \geq \eps^2/4$.
Appealing to~\eqref{eq:rescale}, we find that 

$$ D(u_1,\dots,u_k) = (1+\delta)^k D(w_1,\dots,w_k) \leq (1+\delta)^k c' < v_k. $$

\noindent
This shows that the lemma holds with the choice $c := (1+\delta)^k c'$.
\end{proof}

\begin{proofof}{Proposition~\ref{prop:Nkeps}}
We fix $\eps>0$, let $\delta, c$ be as provided by Lemma~\ref{lem:zevenendertig} and define

$$ N := \left|\left\{ (z_1,\dots,z_k) \in \Zcal^k :
\begin{array}{l} \rho(z_1,\dots,z_k) \leq 1+\delta, \text{ and }\\
z_1,\dots,z_k \text{ not $\eps$-near regular}
\end{array} \right\}\right|. $$

\noindent 
Arguing as in the proof of Lemma~\ref{lem:Iint}, we find:

$$ \begin{array}{rcl}
 \Ee N 
 & = & \displaystyle 
\lambda^k \int_{\eR^d}\dots\int_{\eR^d}
 1_{\tiny\left\{\begin{array}{l} \rho(x_1,\dots,x_k) \leq 1+\delta, \text{ and }\\
x_1,\dots,x_k \text{ not $\eps$-near regular}
\end{array}\right\}} \dd x_1\dots\dd x_k \\[2ex]
& = & \displaystyle 
d^{O(1)} \int_{\eR^k}\dots\int_{\eR^k}
 1_{\tiny\left\{\begin{array}{l} \rho(u_1,\dots,u_k) \leq 1+\delta, \text{ and }\\
u_1,\dots,u_k \text{ not $\eps$-near regular}
\end{array}\right\}} \cdot D(u_1,\dots,u_k)^{d-k}\dd u_1\dots\dd u_k \\[2ex]
& \leq & \displaystyle
d^{O(1)} \cdot \left( \kappa_k 3^k \right)^k \cdot c^{d-k} \\
& = & \displaystyle 
\left( c + o_d(1) \right)^d, 
\end{array} $$

\noindent 
where $c$ is as provided by Lemma~\ref{lem:zevenendertig}, which we've applied in the penultimate line.
For every fixed $\eta>0$ we have

$$ \Pee\left( N \geq \eta f_{d-k}(\Vtyp) \right) 
\leq 
\Pee\left( N \geq \eta \left(\frac{c+v_k}{2}\right)^d \right) 
+ \Pee\left( f_{d-k}(\Vtyp) \leq \left(\frac{c+v_k}{2}\right)^d \right) 
= o_d(1), $$

\noindent
by Theorem~\ref{thm:fass} and Markov's inequality.
This proves $N / f_{d-k}(\Vtyp) \xrightarrow[d\to\infty]{\Pee} 0$.

To conclude the proof, we remind the reader that as explained in the start of the proof of Theorem~\ref{thm:fass}, with 
probability $1-o_d(1)$, every 
$(d-k)$-face of $\Vtyp$ is defined by some $z_1,\dots,z_k$ with $\rho(z_1,\dots,z_k) \leq 1+\delta$.
It follows that, with probability $1-o_d(1)$:

$$ f_{d-k}(\Vtyp) - M_{k,\eps} \leq N. $$

\noindent
In other words 

$$1 - \frac{N}{f_{d-k}(\Vtyp)} \leq \frac{M_{k,\eps}}{f_{d-k}(\Vtyp)} \leq 1, $$ 

\noindent
with probability $1-o_d(1)$. 
The proposition follows.
\end{proofof}

\section{Discussion and suggestions for further work}

We have shown that the inradius, outradius, diameter and mean with of the typical Poisson-Voronoi cell,
after normalization and when the dimension tends to infinity, 
all tend in probability to explicit constants. For the width we've shown non-matching upper and lower bounds.
As already stated in the introduction, we could not prove the following natural conjecture.

\begin{conjecture}
There is a constant $c$ such that 

$$ \frac{\width(\Vtyp)}{\sqrt[d]{\lambda \cdot \vol(B)}} \xrightarrow[d\to\infty]{\Pee} c. $$

\end{conjecture}

The lower bound on the width can probably be improved via a more technical variation on our argument, 
but it seems unlikely to us that it will give a sharp result without significant new ideas.
The argument giving the upper bound exhibits a direction $u \in S^{d-1}$ such that 
$w(u,\Vtyp)$ is small. The direction $u$ is chosen perpendicular to one of the facets of $\Vtyp$.
A priori we see no reason to believe a direction of this type should minimize $w(u,\Vtyp)$ over all 
possible directions.
Let us however point out that if $u$ is the direction that minimizes $w(u,\Vtyp)$ then each 
of the two supporting hyperplanes
perpendicular to $u$ must contain some face and the sum of the dimensions of these two faces
must be at least $d-1$. (Otherwise, a small perturbation of $u$ will yield a direction with even smaller 
width.) Put differently, the union of the two supporting hyperplanes perpendicular to $u$ 
contains at least $d+1$ vertices of $\Vtyp$.
We have not been able to turn this observation into an argument that gives better bounds than the 
ones provided by Proposition~\ref{prop:width}, but perhaps other teams will be able to succeed in doing that.

An important ingredient in our proofs was the observation that, with probability $1-o_d(1)$, all 
vertices of the typical cell have approximately the same norm (namely $1\pm o_d(1)$ under the scaling $\lambda = 1/\kappa_d$
used throughout the paper).
As can for instance be seen from Theorem 1.2 in~\cite{HoermannEtal}, the expected number of 
vertices of $\Vtyp$ is $\exp[ (d/2) \cdot \ln d \cdot (1+o_d(1)) ]$.
It seems natural the compare the behaviour of the typical cell to that of the convex hull of the same number 
of points taken i.i.d.~uniformly at random on $S^{d-1}$.
The latter set-up has been studied by Bonnet and O'Reilly~\cite{BonnetOReilly}. A relevant result in~\cite{BonnetOReilly} 
is Theorem 3.4, which tells us the convex hull of $n=\exp[ (d/2) \cdot \ln d \cdot (1+o_d(1)) ]$
i.i.d.~uniform points on $S^{d-1}$ has the property that, with probability $1-o_d(1)$, each of its facets has 
distance $1\pm o_d(1)$ to the origin.
This is rather different from the behaviour of $\Vtyp$. In the proof of Lemma 14 we have exhibited 
the existence (with probability $1-o_d(1)$) of some facet with distance $1/2\pm o_d(1)$ to the origin. 
A straightforward adaptation of the argument shows that for all $1/2\leq r\leq 1$ there exist 
(with probability $1-o_d(1)$) some facet of $\Vtyp$ whose distance to the origin is $r\pm o_d(1)$.
So, in some sense the typical Poisson-Voronoi cell is less ``symmetric'' or ``regular'' than the 
convex hull of $n$ i.i.d.~points on the unit sphere, if $n$ is taken comparable to the (expected) number of
vertices of $\Vtyp$.

Another natural direction for further research is to try and obtain more precise results on the behaviour of the 
faces of $\Vtyp$. E.g.~give more detailed quantitative bounds on the diameter of (almost all) $k$-faces.
In the same line of questioning, one may want to determine more precise estimates on the $f_{d-k}(\Vtyp)$ than 
Theorem~\ref{thm:diambigface} provides. As mentioned earlier, in a forthcoming article we plan to 
provide more detailed asymptotics.

As the reader may have noticed -- and if not can readily check -- all the error probabilities
$o_d(1)$ obtained in our proofs were in fact exponentially small in $d$ or much smaller in many cases, with 
the exception of the lower bound in Proposition~\ref{prop:width}.
We've always considered the probability that the considered quantity (inradius, outradius, etc.) differs
by a constant $\eps>0$ from the target value. A natural set of follow-up questions is to see 
how fast we can let $\eps$ tend to zero as $d\to\infty$. 
In the same vein, it would be natural to find or bound the variances or perhaps even find a normalization
that yields non-trivial limiting distributions.

\subsection*{Acknowledgements}

We thank Gilles Bonnet, Pierre Calka and Anna Gusakova for helpful discussions and pointers to the literature.
We thank Ali Khezeli for making us aware of the existence of~\cite{AlishahiThesis}.

\bibliographystyle{plain}
\bibliography{IrlbeckKabluchkoMuller}

\appendix

\section{The Hausdorff distance between $\Vtyp$ and any ball is large\label{sec:Hausdorff}}

Here we substantiate the claim from the introduction that, with probability $1-o_d(1)$, the Hausdorff distance 
between the typical cell $\Vtyp$ and any ball is at least 
a constant times the diameter of $\Vtyp$.
For completeness let us first recall that the {\em Hausdorff distance} between sets $X, Y \subseteq \eR^d$ is 
given by:

$$ d_H(X,Y) := \max\left( \sup_{x\in X} \inf_{y\in Y} \norm{x-y}, \sup_{y \in Y} \inf_{x\in X} \norm{x-y} \right). $$

\noindent
We note that the Hausdorff distance satisfies

$$ 2\cdot d_H(X,Y) \geq |\diam(X)-\diam(Y)|, |\width(X)-\width(Y)|. $$

\noindent
If $Y$ is a ball then $\diam(Y)=\width(Y)$.
So for $X \subseteq \eR^d$ arbitrary 
and $Y \subseteq \eR^d$ a ball we must have 

$$ \diam(X)-2\cdot d_H(X,Y)\leq \diam(Y)=\width(Y) \leq w(X) + 2\cdot d_H(X,Y), $$

\noindent
giving  

$$d_H(X,Y) \geq (\diam(X)-w(X))/4.$$

Hence, by Theorem~\ref{thm:main} and Proposition~\ref{prop:width}, with probability $1-o_d(1)$, the typical 
cell $X=\Vtyp$ satisfies:

$$ d_H(\Vtyp, Y) \geq (1/8-o_d(1))\cdot \sqrt[d]{\lambda \cdot \vol(B)} 
= (1/16-o_d(1))\cdot\diam(\Vtyp), $$

\noindent 
for every ball $Y$.

\end{document}